\documentclass[12pt]{article}

\usepackage{amsmath,amssymb,amsthm,euscript}
\newtheorem{theorem}{Theorem}
\newtheorem{lemma}{Lemma}

\theoremstyle{definition}
\newtheorem{definition}{Definition}

\renewenvironment{proof}{
\par\noindent{\it Proof.}} { \mbox{}\hfill $\blacksquare$ \par }
\newcommand\fracstyle{}

\normalbaselines \setbox\strutbox\hbox{\vrule width 0pt height 6.4pt
depth 1.6pt}

\newcommand\GCD{\mathop{\rm gcd}}

\newcommand\sign{\mathop{\rm sign}\nolimits }

\renewcommand\mod{\allowbreak\,\mathord{\rm mod}\allowbreak\, }

\newcommand\SSym{\mathop{\rm Sym}\nolimits }
\newcommand\per{\mathop{\rm per}\nolimits }

\newcommand\Tr{\mathop{\rm Tr}\nolimits }

\newcommand\Inj{\mathop{\rm Inj}\nolimits }
\newcommand\Coef{\mathop{\rm Coef}}

\newcommand{\oalp}{\overline\alpha}
\newcommand{\obet}{\overline\beta}
\newcommand{\ogam}{\overline\gamma}
\newcommand{\ls}{\left(}
\newcommand{\rs}{\right)}
\DeclareMathOperator{\dif}{d}
\newcommand{\bbZ}{\mathbb{Z}}
\newcommand{\vep}{\varepsilon}
\newcommand{\suml}[2]{\sum\limits_{#1}^{#2}}
\DeclareMathOperator{\Card}{Card}
\newcommand\rows[1]{\vcenter{\baselineskip=7pt\lineskip=1pt\lineskiplimit1pt
\halign{\hfil$\scriptstyle##$\hfil\cr #1}}}

\newcommand\vt[2]{\vtop{\hsize=#1\hsize \hangafter1\hangindent\parindent
\noindent #2}}
\newcommand\Vt[1]{\vtop{\halign{##\hfil\cr #1}}}
\newcommand\spusk[1]{\vbox to 120pt{#1}}

\newcommand\Binom[2]{\LR(){\displaystyle{\hbox{\fracstyle$#1$\vrule
width 0pt height \ht\strutbox}\atop\hbox{\fracstyle$#2$\vrule
width 0pt depth \dp\strutbox}}}}

\mathchardef\R"71B3
\mathchardef\Z"71B4
\mathchardef\C"71B2
\mathchardef\Q"71B1
\mathchardef\N"71B0
\mathchardef\k"71B9

\newcommand\syscase[1]{\left\{\,\vcenter{\halign{$##$\hfil\cr
#1}}\right.}

\renewcommand\matrix[1]{\null\,\vcenter{\fns\normalbaselines\mathsurround0pt
    \ialign{\hfil$##$\hfil&&\enspace\hfil$##$\hfil\crcr
      \mathstrut\crcr\noalign{\kern-\baselineskip}
      #1\crcr\mathstrut\crcr\noalign{\kern-\baselineskip}}}\,}

\newcommand\dfill{\cleaders\hbox to 10pt{\hss.\hss}\hfill }

\newcommand\LR[3]{\setbox0\hbox{$#3$}\setbox1\hbox{$\left#1\vcenter{\copy0}
\right#2$}\dimen200\ht1\advance\dimen200 by -\ht0
\dimen201\ht1\advance\dimen201 by \dp1
\mathord{\mathopen{\lower\dimen200\hbox{$\left#1\vcenter to
\dimen201{}\right.$}}\copy0\mathclose{\lower\dimen200
\hbox{$\left#2\vcenter to \dimen201{}\right.$}}}}

\begin{document}


\vspace{30mm} \spusk{\leftskip0pt plus 1fil\parindent0pt
\rightskip0pt plus 1fil\parfillskip0pt \bf\boldmath Rationality of
 generating functions of rook polynomials and \\
permanents of Kronecker products of Toeplitz matrices and \\
circulants with the matrix~$J_k$ and their evaluation. I.
\bigskip\centerline{\it A.M. Kamenetsky} \vss}

\begin{abstract}
In this paper we give a generalization created by author theory of
the rook polynomials \break and permanents of circulants, Toeplits
matrices and their submatrices. Let
$$
T_n^k = (b_{i,j})_{1 \leqslant i,j \leqslant n}, \qquad P_n =
(c_{i,j})_{1 \leqslant i,j \leqslant n} ,
$$
$$
b_{i,j}= \begin{cases} 1, &\text{ if $j-i=k,$} \\
0, &\text{if $j-i \ne k,$}
\end{cases} \quad
c_{i,j}=\begin{cases} 1, &\text{if $j-i \equiv 1 (\mod n) $} \\
0,&\text{$j-i \not\equiv 1 (\mod n).$}
\end{cases}
$$
Let $i^{<l>} = i,i,\ldots,i$ ($l$~times), let
$$
G_t^{[k]} = \{ (1^{l_1}, \ldots, t^{l_t})| \, 0 \leqslant l_i
\leqslant k, 1 \leqslant i \leqslant t \}, \; G_{r,t}^{[k]} = \{
(1^{<l_1>}, \ldots, t^{<l_t>} \in G_t^{[k]})| \sum_{i=1}^t l_i = r
\},
$$
let $R(x;A)$~ be the rook polynomial of a rectangular matrix $A$
over a commutative ring with unity. On the sets $G_t^{[k]}$ and
$G_{r,t}^{[k]}$ we define the square matrices $K_t^{[k]}(a_0, a_1,
\ldots, a_t)$ and $\prod_{r,t}^{[k]}(a_0, a_1, \ldots , a_t)$,
respectively. If $X$~is a finite set of cardinality $n$ then we
denote by $\mathbf{L}$ (in the context of the concrete set~$X$ under
consideration) a fixed bijection from $X$ onto $N_n = \{1, 2,
\ldots, n \}.$ Let $J_k$~denote the $k \times k$ matrix of $1'$s.
Let $\{ \overline \alpha \}$~be a multiset composed from components
of a vector $\overline \alpha.$

The basic results obtained in this paper are following. Let $0
\leqslant r \leqslant t,$ $a_{-r}, a_{-r+1}, \ldots \break \ldots,
a_{-r+t} $ be elements in a commutative ring with unity. Then for all
$n \geqslant 1$
\begin{multline*}
R \ls x;\ls \sum_{i=0}^{t} a_{-r+i} T_n^{(-r+i)} \rs \otimes J_k  \rs =\\
= \sum_{\substack {\overline \gamma \in \bigcup_{l=rk}^{kt}
G_{l,t}^{[k]}
\\ \{ \overline \gamma \} \supseteq \{ 1^{<k>}, 2^{<k>}, \ldots,
r^{<k>}\}}} (K_t^{[k]}( a_{-r}x, a_{-r+1}x, \ldots, a_{-r+t}x))^n [
\mathbf{L}(1^{<k>}, \ldots, r^{<k>})|
\mathbf{L}(\overline \gamma) ], \\
R \ls x;\ls \sum_{i=0}^{t} a_{-r+i} P_n^{-r+i} \rs \otimes J_k \rs =
\Tr((K_t^{[k]}( a_{-r}x, a_{-r+1}x, \ldots, a_{-r+t}x))^n), \\
 \per \ls \ls \sum_{i=0}^{t} a_{-r+i} T_n^{(-r+i)} \rs \otimes J_k \rs =
\\=
 \ls \Pi_{rk,t}^{[k]}( a_{-r}, a_{-r+1}, \ldots, a_{-r+t})\rs^n
[\mathbf{L}(1^{<k>}, \ldots, r^{<k>})| \mathbf{L}(1^{<k>}, \ldots, r^{<k>})],\\
\per \ls \ls \sum_{i=0}^{t} a_{-r+i} x^i P_n^{-r+i} \rs \otimes J_k
\rs = \sum_{l=0}^{kt} \ls \Tr \ls \ls \Pi_{l,t}^{[k]}( a_{-r},
a_{-r+1}, \ldots, a_{-r+t}) \rs^n \rs \rs x^{ln}.
\end{multline*}
\end{abstract}

It is well known that entire classical theory of enumeration of
permutations with restricted positions, ascending still to L.~Euler
(1708), P.~R.~Montmort (1713), P.~G.~Tait, A.~Cayley, T.~Muir
(1878), E.~Lucas (1891), J.~Touchard (1934) and I.~Kaplansky (1943),
to reduce to evalution of rook polynomials and permanents of $(0,1)$
Toeplitz matrices and (0,1)~circulant matrices.

Really, let $\mathbb {Z}$ be the ring of rational integers, $S$ be a
finite subset of the set  $\bbZ$, $\SSym(n)$ be the symmetric group
of substitutions of the set ${1,2,\dots,n}$,
\begin{gather*}
A=\{\sigma\in\SSym(n)\mid \sigma(i)-i\in S \mbox{ for all } i, \, 1\leqslant i\leqslant
n\}, \\
B=\{\sigma\in\break\SSym(n)\mid \sigma(i)-i \notin S \mbox{ for
all } i, \, 1\leqslant i\leqslant n\}, \\
C=\{\sigma\in\SSym(n)\mid \sigma(i)-i\in S+n\bbZ \mbox{ for all }
i, \, 1\leqslant i\leqslant n\},\\
D=\{\sigma\in\SSym(n)\mid \sigma(i)-i\notin S+n\bbZ \mbox{ for all
} i, \, 1\leqslant i\leqslant n\}.
\end{gather*}
Then $ |A|=\per \Bigl(\, \sum\limits_{i\in S} T_n^{(i)} \Bigr), \:
|B|=\per \Bigl(J_n-\sum\limits_{i\in S}T_n^{(i)} \Bigr)$ for all $n
\geqslant 1$ and $ |C|=\per \Bigl(\,\sum\limits_{i\in
S}P_n^{i}\Bigr),$  $|D|= \break =\per\Bigl( J_n-\sum\limits_{i\in
S}P_n^{i}\Bigr), $ for all $n\geqslant(\max S-\min S)+1$, where
$T_n^{(i)}$ and $P_n$ defined down. Further, if $A$ is an $m\times
n$ matrix with $m\leqslant n$ and the rook polynomial of $A$ is
$R(x;A) =1+\sum\limits_{i=1}^m r_i(A) x^i$, then $\per(y J_{m,n} -
A)=\sum\limits_{i=0}^m(-1)^i\binom{n-i}{m-i}(m-i)!\,r_i(A) \cdot
y^{m-s}$, where $r_0(A)=1$ and $J_{m,n}$ denote the $m\times n$
matrix of 1's.

The present paper is a generalization being created by the author
theory ([1], [2], [3]) of the permanents and the rook polynomials of
general circulants, Toeplitz matrices and their submatrices.

Let $\Inj(S,X)$~ be the set of injections of a set $S$ into a
set~$X$. Let $A=(a_{i,j})_{\rows{1\leqslant i\leqslant m\cr
1\leqslant j\leqslant n\cr}} $ be a rectangular matrix over a
commutative ring with unity, let $m\leqslant n$, let
\begin{equation}
R(x;z;A)=1+\sum\limits_{k=1}^mx^k\sum\limits_{\rows{S\subseteq
N_m\cr |S|=k\cr}}\sum\limits_{\sigma\in \Inj(S,N_n)}z^{|\sigma|}
\prod\limits_{i\in S}a_{i,\sigma(i)}
\end{equation}
be the cycle rook polynomial of the matrix~$A$, let
$\per(z;A)=\sum\limits_{\sigma\in\Inj(N_m,N_n)}z^{|\sigma|}
\prod\limits_{i\in N_m}a_{i,\sigma(i)}$ be the cycle permanent of
the matrix~$A$, let $R(x;A)=R(x;1;A)$ and $\per(A)=\per(1;A)$, where
$|\sigma|$ is the number of cycles of the injective
mapping~$\sigma$, and let $N_m=\{1,2,\dots,m\}$. Let
$T_n^{(k)}=(b_{i,j})_{1\leqslant i,j\leqslant n}$, and
$P_n=(c_{i,j})_{1\leqslant i,j\leqslant n}$, where
$$
b_{i,j}=\begin{cases}
1,&\text{ if $j-i=k$,} \\
0,&\text{if $j-i\not=k$,}
\end{cases}
\qquad c_{i,j}=\begin{cases}
1,& \text{ if $j-i\equiv1\pmod n$,}\\
0,& \text{ if $j-i\not\equiv1\pmod n$.}
\end{cases}
$$

Let $m_{\overline\alpha}(i)$ be the number of components of a vector
$\overline\alpha$ equal to~$i$, let $i^{\langle
l\rangle}=i,i,\dots,i$ ($l$ times), and let
\begin{gather*}
G_{\smash t}^{[k]} = \{(1^{\langle l_1\rangle},2^{\langle
l_2\rangle},\dots, t^{\langle l_t\rangle})\mid 0\leqslant
l_i\leqslant k,\allowbreak 1\leqslant i\leqslant t\}, \\
G_{\smash{r,t}}^{[k]} =\Bigl\{(1^{\langle l_1\rangle},2^{\langle
l_2\rangle}, \dots,t^{\langle l_t\rangle})\in G_{\smash t}^{[k]}
~\Big|~ \sum\limits_{i=1}^tl_i=r\Bigr\}, \, 0\leqslant r\leqslant
kt.
\end{gather*}
Then
$$
\displaylines{
|G_{\smash{r,t}}^{[k]}|=\Coef\limits_{x^r}(1+x+\dots+x^k)^t = \hfill\cr
\hfill =\sum\limits_{\rows{\lambda_0+\lambda_1+\dots+\lambda_k=t\cr
\lambda_1+2\lambda_2+\dots+k\lambda_k=r\cr}}\frac{t!}{\lambda_0!\,
\lambda_1!\dots\lambda_k!} = \smash{\sum\limits_{l=0}^{\textstyle\left[{r\over
k+1}\right]}(-1)^l\binom tl\binom{t+r-l(k+1)-1}{t-1},}\cr
\noalign{\bigskip}
|G_{\smash t}^{[k]}|=(k+1)^t.\cr}
$$
On the set $G_{\smash{r,t}}^{[k]}$ we define the square matrix
$\Pi_{\smash{r,t}}^{[k]}(a_0,a_1,\dots,a_t)=(a_{\overline\alpha,
\overline\beta})_{\overline\alpha,\overline\beta\in G_{\smash{r,t}}
^{[k]}}$ as follows: $\Pi_{\smash{0,0}}^{[k]}=k!\,a_0^k$; if
$t\geqslant1$, $\overline\alpha=(1^{\langle l_1\rangle},2^{\langle
l_2\rangle},\dots, t^{\langle l_t\rangle})\in
G_{\smash{r,t}}^{[k]}$, then
$$
a_{\overline\alpha,\overline\beta}=\syscase{\Vt{$k!\Binom k
{l_t+\mathop{\Sigma}
\limits_{i=1}^{t-1}p_i}\binom{l_1}{p_1}\binom{l_2}{p_2}\dots
\binom{l_{t-1}}{p_{t-1}}a_{\smash0}^{k-l_t-\mathop{\Sigma}
\limits_{i=1}^{t-1} p_i}a_{\smash1}^{p_1}a_{\smash2}^{p_2}\dots
a_{\smash{t-1}}^{p_{t-1}} a_{\smash t}^{l_t}$,\cr\indent if
$\overline\beta=\LR(){1^{\left\langle
l_t+\mathop{\Sigma}\limits_{i=1}^{t-1}p_i\right\rangle}, 2^{\langle
l_1-p_1\rangle},3^{\langle l_2-p_2\rangle},\dots, t^{\langle
l_{t-1}-p_{t-1}\rangle}}$, $0\leqslant p_i\leqslant l_i$,\cr\indent
$1\leqslant i\leqslant t-1$, $\sum\limits_{i=1}^{t-1}p_i\leqslant
k-l_t$,\cr}\cr 0 \quad \mbox{otherwise.}\cr}
$$
On the set $G_{\smash{t}}^{[k]}$ we  define the square matrix
$K_{\smash{t}}^{[k]}(a_0,a_1,\dots,a_t)=(a_{\overline\alpha,
\overline\beta})_{\overline\alpha,\overline\beta\in G_{\smash{t}}
^{[k]}}$ as follows: $K_{\smash0}^{[k]}(a_0)=\sum\limits_{d=0}^kd!
\binom kd^2a_0^d$.  If $t\geqslant 1$, $\overline\alpha=(1^{\langle
l_1\rangle},2^{\langle l_2\rangle},\dots, t^{\langle l_t\rangle})\in
G_{\smash{t}}^{[k]}$, then
$$
a_{\overline\alpha,\overline\beta}=\syscase{\Vt{$\binom{l_1}{p_1}
\binom{l_2}{p_2}\dots\binom{l_{t-1}}{p_{t-1}}\Binom k{v-\mathop{
\Sigma}\limits_{i=1}^{t-1}p_i}a_{\smash0}^{v-\mathop{\Sigma}\limits
_{i=1}^{t-1}p_i}a_{\smash1}^{p_1}a_{\smash2}^{p_2}\dots
a_{\smash{t-1}} ^{p_{t-1}}\times$\cr\indent$ \times
\sum\limits_{d=0}^{\min(l_t,k-v)}(d+v)!\binom k{d+v}\binom
{l_t}da_t^d$,\cr\indent if $\overline\beta=\LR(){1^{\left\langle
k-v+ \mathop{\Sigma}\limits_{i=1}^{t-1}p_i\right\rangle},2^{\langle
l_1-p_1 \rangle},3^{\langle l_2-p_2\rangle},\dots,t^{\langle
l_{t-1}-p_{t-1} \rangle}}$,\cr\indent $0\leqslant p_i\leqslant l_i$,
$1\leqslant i\leqslant t-1$, $\sum\limits_{i=1}^{t-1}p_i\leqslant
v\leqslant k$,\cr}\cr \noalign{\medskip}
\rlap0\indent\hbox{otherwise.}\cr}
$$

Let $\{\overline\alpha\}$ be a multiset, consisting from the
components of a~vector~$\overline\alpha$. If $X$ is a finite set of
cardinality~$n$ then we denote by symbol~$\mathbf{ L }$ (in the
context of the concrete set~$X$ under consideration) a fixed
bijection from~$X$ onto~$N_n$. Let $J_k$ be the square matrix of
order~$k$, with all elements equal to~1.
If $\overline\alpha=(x_1,x_2,\dots,x_m) \in \bbZ^m$, $k,n
\in\bbZ,$ then, by definition,
$k\overline\alpha=(kx_1,kx_2,\dots,kx_m)$,
$n+\overline\alpha=(n+x_1,n+x_2,\dots,n+x_m)$.

\begin{theorem} Let $t\geqslant 0$, let $a_0$, $a_1$,~\dots,~$a_t$ be elements
in a commutative ring with unity, let $n\geqslant1$ and
$k\geqslant1$, let $\overline\beta\in G_{\smash{\min(n,t)}}^{[k]}$,
let $m=\min(n,t)$. Then for all $\overline\alpha\in G_{\smash
t}^{[k]}$
$$
\displaylines{ \textstyle
R\left(x;\left(\left(\sum\limits_{i=0}^ta_iT_{\smash{n+t}}^{(i)}
\right)\otimes J_k\right)[N_{nk}\mid k( 1^{\langle
k-m_{\overline\beta}(1)\rangle}, 2^{\langle
k-m_{\overline\beta}(2)\rangle},\dots, m^{\langle
k-m_{\overline\beta}(m)\rangle}),\right.\hfill\cr \hfill
mk+1,mk+2,\dots,nk,nk+k\overline\alpha]\bigg)= \cr \hfill
=\displaystyle \left(\prod\limits_{i=1}^{\min(n,t)}\binom
k{m_{\overline\beta}(i)}^{-1}\right)
\sum\limits_{\rows{\overline\gamma \in G_{\smash t}^{[k]}\cr
\{\overline\gamma\}\supseteq\{\overline
\beta\}\cr}}\prod\limits_{i=1}^{\min(n,t)}\binom{m_{\overline
\gamma}(i)}{m_{\overline\beta}(i)}\cdot(K_{\smash t}^{[k]}(a_0x,
a_1x,\dots a_tx))^n[\mathbf{ L }(\overline\alpha)\mid \mathbf{ L
}(\overline \gamma)].\cr}
$$
\end{theorem}

\begin{theorem}
Let $t\geqslant0$, let $a_0$, $a_1$,~\dots,~$a_t$ be elements in a
commutative ring with unity, let $n\geqslant1$ and $k\geqslant1$,
let $0\leqslant r\leqslant k\min(n,t)$, let $\overline\beta\in
G_{\smash{r,\min(n,t)}}^{[k]}$, let $m=\min(n,t)$. Then for all
$\overline\alpha\in G_{\smash{r,t}}^{[k]}$
$$
\displaylines{ \textstyle
\per\left(\left(\left(\sum\limits_{i=0}^ta_iT_{\smash{n+t}}
^{(i)}\right)\otimes J_k\right)\bigl[N_{nk}\mid k( 1^{\langle
k-m_{\overline\beta}(1)\rangle}, 2^{\langle
k-m_{\overline\beta}(2)\rangle},\dots, m^{\langle
k-m_{\overline\beta}(m)\rangle}),\right.\hfill\cr \hfill
mk+1,mk+2,\dots,nk,nk+k\overline\alpha\bigr]\biggr)=\cr
\hfill\displaystyle=\left(\prod\limits_{i=1}^{\min(n,t)}\binom k
{m_{\overline\beta}(i)}^{-1}\right)(\Pi_{\smash{r,t}}^{[k]}(a_0,a_1,
\dots,a_t))^n[ \mathbf{ L } (\overline\alpha)\mid \mathbf{ L }
(\overline\beta)].\cr}
$$
\end{theorem}
\begin{theorem}
Let $0\leqslant r\leqslant t$, let $a_{-r}$,
$a_{-r+1}$,~\dots,~$a_{-r+t}$ be  elements  in a commutative
ring with unity. Then for all $n\geqslant1$
$$
\displaylines{ \textstyle
R\left(x;\left(\sum\limits_{i=0}^ta_{-r+i}T_n^{(-r+i)}
\right)\otimes J_k\right)= \hfill\cr\hfill\textstyle
=\sum\limits_{\rows{\overline\gamma\in
\mathop{\cup}\limits_{l=rk}^{kt}G_{\smash{l,t}}^{[k]}\cr
\{\overline\gamma\}\supseteq\{1^{\langle k\rangle},2^{\langle
k\rangle},\dots,r^{\langle k\rangle}\}\cr}}(K_{\smash t}^{[k]}
(a_{-r}x,a_{-r+1}x,\dots,a_{-r+t}x))^n[ \mathbf{ L } (1^{\langle
k\rangle}, 2^{\langle k\rangle},\dots,r^{\langle k\rangle})\mid
\mathbf{ L } (\overline \gamma)],\cr \noalign{\bigskip}
\textstyle\per\left(\left(\sum\limits_{i=0}^ta_{-r+i}T_n^{(-r+i)}
\right)\otimes J_k\right) = \hfill\cr\hfill =
(\Pi_{\smash{rk,t}}^{[k]}(a_{-r},a_{-r+1},\dots, a_{-r+t}))^n[
\mathbf{ L } (1^{\langle k\rangle},2^{\langle k\rangle},\dots,
r^{\langle k\rangle})\mid \mathbf{ L } (1^{\langle
k\rangle},2^{\langle k \rangle},\dots,r^{\langle k\rangle})],\cr
\noalign{\bigskip} \textstyle
R\left(x;\left(\sum\limits_{i=0}^ta_{-r+i}P_n^{-r+i} \right)\otimes
J_k\right)=\Tr((K_{\smash t}^{[k]}(a_{-r}x,a_{-r+1}x,
\dots,a_{-r+t}x))^n),\cr \noalign{\bigskip} \textstyle
\per\left(\left(\sum\limits_{i=0}^ta_{-r+i}x^iP_n^{-r+i}
\right)\otimes
J_k\right)=\sum\limits_{l=0}^{kt}(\Tr((\Pi_{\smash{l,t}}
^{[k]}(a_{-r},a_{-r+1},\dots,a_{-r+t}))^n))x^{ln}.\cr}
$$
\end{theorem}

\begin{theorem}
 $\per((a_0I_n+a_1P_n)\otimes J_k)=(k!)^n\sum\limits
_{l=0}^k\binom kl^n(a_0^{k-l}a_1^l)^n$ for all $n\geqslant1$.
\end{theorem}

\begin{proof}
Since $\Pi_{\smash{l,1}}^{[k]}(a_0,a_1)=k!\binom kla_0^{k-l}a_1^l$,
$0\leqslant l\leqslant k$,
 then from the theorem 3 it follows that
$$
\per((a_0I_n+a_1P_n)\otimes J_k)=\sum\limits_{l=0}^k\left(k!\binom
kl a_0^{k-1}a_1^l\right)^n.
$$
\end{proof}
Let $t\geqslant1$ and $0\leqslant r\leqslant t$, let $a_{-r}$,
$a_{-r+1}$,~\dots,~$a_{-r+t}$ be elements in a commutative ring with
unity, let $a_{{1\over2k}}=1$, let
$W_{\smash{r,t}}^{[k]}=\Bigl\{(\overline
\alpha,\overline\beta)\in\Bigl\{-r,-r+1,\dots,-r+t,\frac1{2k}\Bigr\}
^{kt}\times\{0,1,\dots,k-1\}^{kt}~|\break
\overline\alpha=(\alpha_0,\alpha_1, \dots,\alpha_{kt-1})$,
$\overline\beta=(l_0,l_1,\dots,l_{kt-1})$, the mapping $ks+i\to
k(s+\alpha_{ks+i})+l_{ks+i}$, $0\leqslant s\leqslant t-1$,
$0\leqslant i\leqslant k-1$, of the set $\{0,1,\dots,kt-1\}$ is
injective; if $\alpha_{ks+i}=\frac1{2k}$, then $l_{ks+i}=i\Bigr\}$,
let $ V_{\smash{r,t}}^{[k]}
=\Bigl\{(\overline\alpha,\overline\beta)\in\{-r,-r+1,
\dots,-r+t\}^{kt}\times\{0,1,\dots,k-1\}^{kt} \mid
\overline\alpha=(\alpha_0,\alpha_1,\dots,\alpha_{kt-1}),
\overline\beta= \break =(l_0,l_1,\dots, l_{kt-1}),$ the mapping $
ks+i\to k(s+\alpha_{ks+i})+l_{ks+i}, 0\leqslant s\leqslant t-1,
0\leqslant i\leqslant k-1,$ of the set $\{0,1,\dots,kt-1\}$ is
injective $\Bigr\}. $
Let
$A_{\smash{r,t}}^{[[k]]}(a_{-r},a_{-r+1},\dots,a_{-r+t})=(a_{(\overline\alpha,\overline\beta),(\overline\varepsilon,\overline\rho)})
_{(\overline\alpha,\overline\beta),(\overline\varepsilon,\overline\rho)
\in W_{\smash{r,t}}^{[k]}}$   be the square matrix of order
$|W_{\smash{r,t}}^{[k]}|$ with the elements
$a_{(\overline\alpha,\overline\beta),(\overline\varepsilon,\overline\rho)}$,
where $\overline\alpha=(\alpha_0,\alpha_1,\dots,\alpha_{kt-1})$,
$\overline\beta=(l_0,l_1,\dots,l_{kt-1})$, $\overline\varepsilon =
(\varepsilon_0,\varepsilon_1,\dots,\varepsilon_{kt-1})$,
$\overline\rho = (p_0,p_1,\dots,p_{kt-1})$ defined as follows:
$$
a_{(\overline\alpha,\overline\beta),(\overline\varepsilon,\overline\rho)}=
\syscase{\vt{.8}{$a_{\alpha_0}a_{\alpha_1}\dots a_{\alpha_{k-1}}$,
if $(\alpha_k,\alpha_{k+1},\dots,\alpha_{kt-1})=(\varepsilon_0,
\varepsilon_1,\dots,\varepsilon_{k(t-1)-1})$, \\ $(l_k,
l_{k+1},\dots,l_{kt-1})=(p_0,p_1,\dots,p_{k(t-1)-1})$ and the
mapping \\
$ks+i\to k(s+\alpha_{\smash {ks+i}}')+l_{\smash{ks+i}}'$,
$0\leqslant s\leqslant t$, $0\leqslant i\leqslant k-1$, of the set
$\{0,1,2,\dots,k(t+1)-1\}$ is injective, where $\alpha_{\smash
{ks+i}}'=\alpha_{ks+i}$, $l_{\smash{ks+i}}'=l_{ks+i}$, if
$0\leqslant s \leqslant t-1$, $0\leqslant i\leqslant k-1$, and \\
$(\alpha_{\smash{kt}}',\alpha_{\smash
{kt+1}}',\dots,\alpha_{\smash{k(t+1)-1}}')= (\varepsilon_{k(t-1)},
\varepsilon_{k(t-1)+1},\dots,\varepsilon_{kt-1}), \\
(l_{\smash{kt}}',
l_{\smash{kt+1}}',\dots,l_{\smash{k(t+1)-1}}')=(p_{k(t-1)},
p_{k(t-1)+1},\dots,p_{kt-1})$,}\cr \noalign{\medskip}
\rlap0\indent\hbox{ otherwise.}\cr}
$$
Matrix $A_{\smash{r,t}}^{[k]}(a_{-r},a_{-r+1},\dots,a_{-r+t})$
defined absolutely identically with help of the set
$V_{\smash{r,t}}^{[k]}$.

\begin{theorem} 
$$
\displaylines{\textstyle
R\left(x;\left(\sum\limits_{i=0}^ta_{-r+i}P_n^{-r+i}\right)\otimes
J_k\right)
=\Tr((A_{\smash{r,t}}^{[[k]]}(a_{-r}x,a_{-r+1}x,\dots,a_{-r+t}x))^n)=
\hfill\cr\hfill
=\Tr((A_{\smash{0,t}}^{[[k]]}(a_{-r}x,a_{-r+1}x,\dots,a_{-r+t}x))^n),\cr
\noalign{\bigskip}
\textstyle\per\left(\left(\sum\limits_{i=0}^ta_{-r+i}P_n^{-r+i}\right)
\otimes J_k\right) =
\Tr((A_{\smash{r,t}}^{[k]}(a_{-r},a_{-r+1},\dots,a_{-r+t}))^n)=
\hfill\cr\hfill=
\Tr((A_{\smash{0,t}}^{[k]}(a_{-r},a_{-r+1},\dots,a_{-r+t}))^n),
\quad n\geqslant1.\cr}
$$
\end{theorem}

\begin{theorem}
Let $k \geqslant 1$ and $0 \leqslant r \leqslant t$, let $a_{-r},
a_{-r+1}, \dots, a_{-r+t}$ be elements in a commutative ring with
unity, let $(\overline{ \vep}, \overline{ \rho}) \in V_{r,t},$ let
$\overline{\vep} = (\vep_0, \vep_1, \dots, \vep_{kt-1})$ and
$\overline{\rho} = (\rho_0, \rho_1, \dots, \rho_{kt-1}),$ let the
mapping $ks+i \to k(s + \vep_{ks+i})+ \rho_{ks+i},$ $0 \leqslant s
\leqslant t-1,$ $0 \leqslant i \leqslant k-1$ be a bijection of the
set $\{0,1,\dots,kt-1 \}$ onto itself. Then for all $n \geqslant 1$
\begin{multline*}
R \biggl( x; \biggl(\sum_{i=0}^k a_{-r+i} T_n^{(-r+i)} \biggr)
\otimes J_k \biggr) = \\
=\sum_{\substack{(\overline{\alpha}, \overline{\beta}) \in
W_{r,t}^{[k]} \\ \overline{\alpha}=(\alpha_0, \alpha_1, \dots,
\alpha_{kt-1}) \\ s+ \alpha_{ks+i} \geqslant 0 \text{ for all } s,
\\ 0 \leqslant s \leqslant t-1, \text{ and } i, 0 \leqslant i
\leqslant k-1.}} \Bigl( A_{r,t}^{[[k]]} (a_{-r} x, a_{-r+1}x, \dots,
a_{-r+t} x)\Bigr)^n [\mathbf{L}((\overline{\alpha},
\overline{\beta})) \mid
\mathbf{L}((\overline{\vep}, \overline{\rho})) ]. \\
\shoveleft{per\biggl(\biggl( \sum_{i=0}^t a_{-r+i} T_n^{(-r+i)}
\biggr) \otimes J_k \biggr) = }\\
\sum_{\substack{(\overline{\alpha}, \overline{\beta}) \in
V_{r,t}^{[k]} \\ \overline{\alpha}=(\alpha_0, \alpha_1, \dots,
\alpha_{kt-1}) \\ s+ \alpha_{ks+i} \geqslant 0 \text{ for all } s,
\\ 0 \leqslant s \leqslant t-1, \text{ and } i, 0 \leqslant i
\leqslant k-1.}} \Bigl( A_{r,t}^{[[k]]} (a_{-r} x, a_{-r+1}x, \dots,
a_{-r+t} x)\Bigr)^n [\mathbf{L}((\overline{\alpha},
\overline{\beta})) \mid \mathbf{L}((\overline{\vep},
\overline{\rho})) ].
\end{multline*}
\end{theorem}
\begin{theorem}
 Let $t \geqslant 1$, $k \geqslant 1$, and $0\leqslant r\leqslant t$. Then
\begin{multline*}
|W_{r,t}^{[k]}| = \sum_{s=1}^{k t} (-1)^{kt-s}
 \sum_{ \bigl\{\{1^{\langle k\rangle},2^{\langle k\rangle},
  \dots, t^{\langle k\rangle}\}=\bigcup\limits_{i=1}^sK_i\bigr\}
 \in \Pi(\{1^{\langle k\rangle},2^{\langle k\rangle},
  \dots, t^{\langle k\rangle}\})}
  \biggl( \prod_{i=1}^s
(|K_i|-1)! \biggr) \times \\ \shoveright{ \times \biggl(
\prod_{i=1}^s \bigl(k(t + 1) + \delta_{\min K_i, \max
K_i } + \min K_i - \max K_i \bigr) \biggr),} \\
\shoveleft{|V_{r,t}^{[k]}| = \sum_{s=1}^{kt} (-1)^{kt-s}
 \sum_{ \bigl\{\{1^{\langle k\rangle},2^{\langle k\rangle},
  \dots, t^{\langle k\rangle}\}=\bigcup\limits_{i=1}^sK_i\bigr\}
 \in \Pi(\{1^{\langle k\rangle},2^{\langle k\rangle},
  \dots, t^{\langle k\rangle}\})}
  \biggl( \prod_{i=1}^s
(|K_i|-1)! \biggr) \times }\\
 \times \biggl( \prod_{i=1}^s \bigl(k(t + 1) + \min K_i - \max
K_i \bigr) \biggr).
\end{multline*}
\end{theorem}
The basis for proof of the theorems 1 --- 3 are the following fundamental
characteristics of the matrices $K_{\smash
t}^{[k]}(a_0,a_1,\dots,a_t)$ and
$\Pi_{r,t}^{[k]}(a_0,a_1,\dots,a_t),$ $0 \leqslant r \leqslant kt$.
\begin{theorem}
 Let $k\geqslant1$ and $t\geqslant2$, let $a_0$, $a_1$,~\dots,~$a_t$ be
 elements in a commutative ring with unity, let $\overline\alpha$,
$\overline\beta\in G_{\smash t}^{[k]}$, let $1\leqslant n\leqslant
t-1$. Assume that $\max\{\overline\beta\}\geqslant n+1$,
$\overline\gamma\in G_{\smash t}^{[k]}$,
$\{\overline\gamma\}\subseteq\{\overline\beta\}$,
$\min\{\overline\gamma\}\geqslant n+1$
and~$\{\overline\alpha\}\not\supseteq\{\overline\gamma-n\}$. Then $
(K_{\smash t}^{[k]}(a_0,a_1,\dots,a_t))^n[ \mathbf{ L }
(\overline\alpha) \mid \mathbf{ L } (\overline\beta)]=0. $
\end{theorem}
\begin{proof}
Use induction on~$n$. Without loss of generality, we can assume that
$a_0$, $a_1$,~\dots,~$a_t$   be independent variables. Let $n=1$.
Assume that
$$
(K_{\smash t}^{[k]}(a_0,a_1,\dots,a_t))[ \mathbf{ L }
(\overline\alpha)~|\break \mathbf{ L } (\overline\beta)]\ne0.
$$
Let $\overline\alpha=(1^{\langle l_1\rangle}, 2^{\langle
l_2\rangle},\dots,t^{\langle l_t\rangle})$. Then from definition of
the matrix $K_{\smash t}^{[k]}(a_0,a_1,\dots,a_t)$ it follows that
there exists $p_1$, $p_2$,~\dots,~$p_{t-1}$,~$v$, such that
$0\leqslant p_i\leqslant l_i$, $1\leqslant i\leqslant t-1$,
$\sum\limits_{i=1}^{t-1}p_i\leqslant v\leqslant k$, and
$$
\overline\beta=\LR(){1^{\left\langle k-v+
\mathop{\Sigma}\limits_{i=1}^{t-1}p_i\right\rangle},2^{\langle
l_1-p_1 \rangle},3^{\langle l_2-p_2\rangle},\dots,t^{\langle
l_{t-1}-p_{t-1} \rangle}}.
$$
Let $\overline\gamma=(2^{\langle q_2\rangle},3^{\langle
q_3\rangle},\dots,t^{\langle q_t\rangle})$. Then $q_i\leqslant
l_{i-1}-p_{i-1}$, $2\leqslant i\leqslant t$,
$\overline\gamma-1=(1^{\langle q_2\rangle},2^{\langle
q_3\rangle},\dots,(t-1)^{\langle q_t\rangle})$. But $q_i\leqslant
l_{i-1}-p_{i-1}\leqslant l_{i-1}$, $2\leqslant i\leqslant t$,
and~therefore $\{{\overline\gamma-1}\}\subseteq\{\overline\alpha\}$.
Obtained contradiction prove validity equality (1) for $n=1$. Let
$2\leqslant n\leqslant t-1$ and $t\ge3$. Assume that the theorem~6
is valid for $n-1$. We show that then
$$
((K_{\smash t}^{[k]}(a_0,a_1,\dots,a_t)[ \mathbf{ L }
(\overline\alpha)\mid \mathbf{ L } (\overline\varepsilon)])\cdot
(K_{\smash t}^{[k]}(a_0,a_1,\dots,a_t))^{n-1}[ \mathbf{ L }
(\overline\varepsilon)\mid \mathbf{ L } (\overline\beta)]=0
$$
for all $\overline\varepsilon\in G_{\smash t}^{[k]}$. Really, let
$\overline\varepsilon\in G_{\smash t}^{[k]}$ and {$(K_{\smash
t}^{[k]}(\thinmuskip5mu minus3mu a_0,a_1,\dots,a_t))^{n-1}[ \mathbf{
L } (\overline\varepsilon)\!\mid\! \mathbf{ L }
(\overline\beta)]\ne0$.} Since $\max\{\overline\beta\}\geqslant
n+1>(n-1)+1$, $\{\overline\gamma\}\subseteq\{\overline\beta\}$,
$\min\{\overline\gamma\}\geqslant n+1>(n-1)+1$ then from validity of
the theorem~6 for $n-1$ it follows that
$\{\overline\varepsilon\}\supseteq\{{\overline\gamma-(n-1)}\}$. But
$\{\overline\alpha\}\not\supseteq\{\overline\gamma-n\} =
\{(\overline \gamma-(n-1))-1\}$ and from proved above validity of
theorem 6 for $n=1$ it follows that $(K_{\smash
t}^{[k]}(a_0,a_1,\dots,a_t)[ \mathbf{ L } (\overline\alpha)~|\break
\mathbf{ L } (\overline\varepsilon)]=0$. Even if $(K_{\smash
t}^{[k]}(a_0,a_1,\dots,a_t))^{n-1}[ \mathbf{ L }
(\overline\varepsilon)\!\mid\! \mathbf{ L } (\overline\beta)]=0$,
then the more so,
$$
((K_{\smash t}^{[k]}(a_0,a_1,\dots,a_t)[ \mathbf{ L }
(\overline\alpha)\mid \mathbf{ L } (\overline\varepsilon)])\cdot
(K_{\smash t}^{[k]}(a_0,a_1,\dots,a_t))^{n-1}[ \mathbf{ L }
(\overline\varepsilon)\mid \mathbf{ L } (\overline\beta)]=0.
$$
Therefore
$$
\displaylines{ (K_{\smash t}^{[k]}(a_0,a_1,\dots,a_t))^n[
\mathbf{ L } (\overline\alpha)\mid \mathbf{ L } (\overline\beta)]=
\hfill\cr \hfill =\textstyle \sum\limits_{\overline\varepsilon\in
G_{\smash t}^{[k]}} ((K_{\smash t}^{[k]}(a_0,a_1,\dots,a_t)[
\mathbf{ L } (\overline\alpha)\mid \mathbf{ L }
(\overline\varepsilon)])\cdot (K_{\smash
t}^{[k]}(a_0,a_1,\dots,a_t))^{n-1}[ \mathbf{ L }
(\overline\varepsilon)\mid \mathbf{ L } (\overline\beta)]=0.\cr}
$$
\end{proof}
\begin{theorem}
Let $k\geqslant 1$, and $t\geqslant 2$, let $a_0$, $a_1$,~\dots,~$a_t$ be
elements in a commutative ring with unity let $\overline\alpha$,
$\overline\beta\in G_{\smash t}^{[k]}$, let $1\leqslant n\leqslant
t-1$. Assume that $\max\{\overline\beta\}\geqslant n+1$,
$\overline\gamma\in G_{\smash t}^{[k]}$,
$\{\overline\gamma\}\subseteq\{\overline\beta\}$,
$\min\{\overline\gamma\}\geqslant n+1$,
and~$\{\overline\alpha\}\supseteq\{\overline\gamma-n\}$. Then
\begin{multline}
(K_{\smash t}^{[k]}(a_0,a_1,\dots,a_t))^n[ \mathbf{ L }
(\overline\alpha)\mid
 \mathbf{ L } (\overline\beta)]= \\ \textstyle
=\left(\prod\limits_{i=1}^t\binom{m_{
\overline\alpha}(i)}{m_{\overline\gamma-n}(i)}\binom{m_{\overline
\beta}(i)}{m_{\overline\gamma}(i)}^{-1}\right)\cdot (K_{\smash
t}^{[k]}(a_0,a_1,\dots,a_t))^n[ \mathbf{ L } (\overline\alpha
\setminus(\overline\gamma-n))\mid \mathbf{ L }
(\overline\beta\setminus \overline\gamma)].
\end{multline}
\end{theorem}
{\it Proof } by induction on~$n$. Let $n=1$,
$\overline\alpha=(1^{\langle l_1\rangle}, 2^{\langle
l_2\rangle},\dots,t^{\langle l_t\rangle})$, $\overline\gamma=(
2^{\langle q_2\rangle},3^{\langle q_3\rangle},\dots,t^{\langle
q_t\rangle})$.

Case 1. $(K_{\smash
t}^{[k]}(a_0,a_1,\dots,a_t))[\mathbf{L}(\overline\alpha)\mid
\mathbf{L}(\overline\beta)]\not=0$. Then there exists $p_1$,
$p_2$,~\dots,~$p_{t-1}$,~$v$ such that $0\leqslant p_i\leqslant
l_i$, $1\leqslant i\leqslant t-1$,
$\sum\limits_{i=1}^{t-1}p_i\leqslant v\leqslant k$
and~$\overline\beta=\Bigl(1^{\left\langle k-v+
\mathop{\Sigma}\limits_{i=1}^{t-1}p_i\right\rangle},2^{\langle
l_1-p_1 \rangle},3^{\langle l_2-p_2\rangle},\dots \break
\dots,t^{\langle l_{t-1}-p_{t-1} \rangle} \Bigr)$. Then from
definition of the matrix $K_{\smash t}^{[k]}(a_0,a_1,\dots,a_t)$ it
follows that
$$
\displaylines{ (K_{\smash t}^{[k]}(a_0,a_1,\dots,a_t))[ \mathbf{ L }
(\overline\alpha)\mid  \mathbf{ L } (\overline\beta)]=
\hfill\cr\hfill
=\binom{l_1}{p_1}\binom{l_2}{p_2}\dots\binom{l_{t-1}}{p_{t-1}}
\Binom k{v-\mathop{\Sigma}\limits_{i=1}^{t-1}p_i}
a_{\smash0}^{v-\mathop{\Sigma}\limits_{i=1}^{t-1}p_i}
a_{\smash1}^{p_1}a_{\smash2}^{p_2}\dots a_{\smash{t-1}}^{p_{t-1}}
\textstyle\sum\limits_{d=0}^{\min(l_t,k-v)}(d+v)!\binom{k}{d+v}\binom
{l_t}da_{\smash t}^d,\cr}
$$
$$
\displaylines{
 (K_{\smash t}^{[k]}(a_0,a_1,\dots,a_t))[ \mathbf{ L }
(\overline\alpha \setminus(\overline\gamma-1))\mid
 \mathbf{ L } (\overline\beta\setminus \overline\gamma)]=
\hfill\cr\qquad = \binom{l_1-q_2}{p_1}\binom{l_2-q_3}{p_2}\dots
\binom{l_{t-1}-q_t}{p_{t-1}} \Binom
k{v-\mathop{\Sigma}\limits_{i=1}^{t-1}p_i}
a_{\smash0}^{v-\mathop{\Sigma}\limits_{i=1}^{t-1}p_i}
a_{\smash1}^{p_1}a_{\smash2}^{p_2}\dots
a_{\smash{t-1}}^{p_{t-1}}\times \hfill\cr\hfill \times
\textstyle\sum\limits_{d=0}^{\min(l_t,k-v)}(d+v)!\binom{k}{d+v}\binom
{l_t}da_{\smash t}^d,\cr}
$$
since $\overline\gamma -1=(1^{\langle q_2\rangle},2^{\langle
q_3\rangle},\dots,(t-1)^{\langle q_t\rangle})$,
$\overline\alpha\setminus(\overline\gamma-1)=(1^{\langle
l_1-q_2\rangle},2^{\langle l_2-q_3\rangle},\dots,(t-1)^{\langle
l_{t-1}-q_t\rangle},t^{\langle l_t\rangle})$,\break
$\overline\beta\setminus \overline\gamma=\LR(){1^{\left\langle k-v+
\mathop{\Sigma}\limits_{i=1}^{t-1}p_i\right\rangle},2^{\langle
l_1-p_1-q_2 \rangle},3^{\langle l_2-p_2-q_3\rangle},\dots,t^{\langle
l_{t-1}-p_{t-1}-q_t\rangle}}$. Therefore
$$
\displaylines{ (K_{\smash t}^{[k]}(a_0,a_1,\dots,a_t))[ \mathbf{
L } (\overline\alpha)\mid  \mathbf{ L } (\overline\beta)]=
\hfill\cr\hfill= \left(\prod\limits_{i=1}^{t-1}
\binom{l_i}{p_i}\binom{l_i-q_{i+1}}{p_i}^{-1}\right) (K_{\smash
t}^{[k]}(a_0,a_1,\dots,a_t))[ \mathbf{ L } (\overline\alpha
\setminus(\overline\gamma-1))\mid  \mathbf{ L }
(\overline\beta\setminus \overline\gamma)].\cr}
$$
From equality $\binom ab\binom bc=\binom ac\binom{a-c}{a-b}$ it
follows that
$\binom{l_i}{p_i}\binom{l_i-p_i}{q_{i+1}}=\binom{l_i}{l_i-p_i}
\binom{l_i-p_i}{q_{i+1}}=
\binom{l_i}{q_{i+1}}
\binom{l_i-q_{i+1}}{p_i}$ and~therefore
$\binom{l_i}{p_i}\binom{l_i-q_{i+1}}{p_i}^{-1}=\binom{l_i}{q_{i+1}}
\binom{l_i-p_i}{q_{i+1}}^{-1}$. Hence
$$
\displaylines{
\prod\limits_{i=1}^{t-1}\binom{l_i}{p_i}\binom{l_i-q_{i+1}}
{p_i}^{-1}=\prod\limits_{i=1}^{t-1}\binom{l_i}{q_{i+1}}\binom{l_i-p_i}
{q_{i+1}}^{-1}=\left(\prod\limits_{i=1}^{t-1}\binom{l_i}{q_{i+1}}\right)
\cdot\prod\limits_{i=1}^{t-1}\binom{l_i-p_i}{q_{i+1}}^{-1}=\hfill\cr
\hfill= \left(
\prod\limits_{i=1}^{t-1}\binom{m_{\overline\alpha}(i)}{m_{\overline\gamma
-1}(i)}\right)\cdot\prod\limits_{i=2}^t\binom{m_{\overline\beta}(i)}
{m_{\overline\gamma}(i)}^{-1} = \left(
\prod\limits_{i=1}^t\binom{m_{\overline\alpha}(i)}{m_{\overline\gamma
-1}(i)}\right)\cdot\prod\limits_{i=1}^t\binom{m_{\overline\beta}(i)}
{m_{\overline\gamma}(i)}^{-1}=\hfill\cr
\hfill=\prod\limits_{i=1}^t\binom{m_{\overline
\alpha}(i)}{m_{\overline\gamma-1}(i)}\binom{m_{\overline\beta}(i)}
{m_{\overline\gamma}(i)}^{-1}.\cr}
$$

Case 2. $(K_{\smash
t}^{[k]}(a_0,a_1,\dots,a_t))[\mathbf{L}(\overline\alpha)\mid
\mathbf{L}(\overline\beta)]=0$. Then there exists no $p_1$,
$p_2$,~\dots,~$p_{t-1}$,~$v$ such that $0\leqslant p_i\leqslant
l_i-q_{i+1}$, $1\leqslant i\leqslant t-1$,
$\sum\limits_{i=1}^{t-1}p_i\leqslant v\leqslant k$,
and~$\overline\beta=\Bigl(1^{\left\langle k-v+
\mathop{\Sigma}\limits_{i=1}^{t-1}p_i\right\rangle},2^{\langle
l_1-p_1 \rangle},3^{\langle l_2-p_2\rangle},\dots \break
\dots,t^{\langle l_{t-1}-p_{t-1} \rangle} \Bigr)$. Therefore for all
$p_1, p_2,\dots,p_{t-1},v$, such that $0\leqslant p_i\leqslant l_i -
q_{i+1}$, $1\leqslant i\leqslant t-1$,
$\sum\limits_{i=1}^{t-1}p_i\leqslant v\leqslant k$, this inequality
$$
\overline\beta\setminus\overline\gamma\ne\LR(.{1^{\left\langle
k-v+\mathop{\Sigma}\limits_{i=1}^{t-1}p_i\right\rangle},2^{\langle
l_1-p_1-q_2 \rangle},}\break 3^{\langle
l_2-p_2-q_3\rangle},\dots,t^{\langle l_{t-1}-p_{t-1}-q_t
\rangle})
$$
is correct. Since
$\overline\alpha\setminus(\overline\gamma-1)=(1^{\langle
l_1-q_2\rangle},2^{\langle l_2-q_3\rangle},\dots,(t-1)^{\langle
l_{t-1}-q_t\rangle},t^{\langle l_t\rangle})$, then from definition
of the matrix $K_{\smash t}^{[k]}(a_0,a_1,\dots,a_t)$ it follows
that
\begin{equation*}
(K_{\smash t}^{[k]}(a_0,a_1,\dots,a_t)) [ \mathbf{ L }
(\overline\alpha\setminus(\overline\gamma-1))\mid  \mathbf{ L }
(\overline\beta\setminus\overline\gamma)] =0.
\end{equation*}
 So, for $n=1$
theorem~7 is proved. Let $t-1\geqslant n\geqslant 2$. We show at
first that if $\overline\rho\in G_{\smash t}^{[k]}$\break and there
exists no $\overline\varepsilon\in G_{\smash t}^{[k]}$ such that
$\{\overline\varepsilon\}\supseteq\{\overline\gamma-n+1\}$
and~$\overline\rho=\overline\varepsilon\setminus({\overline\gamma-n+1})$,
then
$$
(K_{\smash t}^{[k]}(a_0,\break a_1,\dots,a_t)) [ \mathbf{ L }
(\overline\alpha\setminus(\overline\gamma-n))\mid  \mathbf{ L }
(\overline\rho)]=0.
$$
Really, then for certain $i \! \in \!\{\overline\gamma-n+1\}$ \,
$m_{\overline\rho}(i)+m_{\overline\gamma-n+1}(i)\!>k$ and~therefore
$m_{\overline\rho}(i)\!>k-m_{\overline\gamma-n+1}(i)$. Since
$i-1\in\overline\gamma-n$ and~$m_{\overline\gamma-n}(i-1)=
m_{\overline\gamma-n+1}(i)$, then
$$
m_{\overline\alpha\setminus(\overline\gamma-n)}(i-1)\leqslant
k-m_{\overline\gamma-n}(i-1)=k-m_{\overline\gamma-n+1}(i)<m_{\overline
\rho}(i),
$$
i.~e.~$m_{\overline\alpha\setminus(\overline\gamma-n)}(i-1)<$
$m_{\overline\rho}(i)$. Let $\overline\gamma_1=i^{\langle
m_{\overline \rho}(i)\rangle}$. Then
$\{\overline\gamma_1\}\subseteq\{\overline\rho\}$, $i\geqslant
(n+1)-n+1=2$,
$\{\overline\alpha\setminus(\overline\gamma-n)\}\not\supseteq\{\overline
\gamma_1-1\}$ and~from theorem 6 it follows that
$$
(K_{\smash t}^{[k]}(a_0,a_1,\dots,a_t)) [ \mathbf{ L }
(\overline\alpha\setminus(\overline\gamma-n))\mid  \mathbf{ L }
(\overline\rho)]=0.
$$
From the induction hypothesis about validity of theorem 7 for $n-1$
and~from theorem 6 it follows that
$$
\displaylines{ (K_{\smash t}^{[k]}(a_0,a_1,\dots,a_t))^n [
\mathbf{ L } (\overline\alpha)\mid \mathbf{ L } (\overline\beta)]=
\hfill\cr \textstyle =\sum\limits_{\overline\varepsilon\in
G_{\smash t}^{[k]}} (K_{\smash t}^{[k]}(a_0,a_1,\dots,a_t)) [
\mathbf{ L } (\overline\alpha)\mid \mathbf{ L }
(\overline\varepsilon)]\cdot (K_{\smash
t}^{[k]}(a_0,a_1,\dots,a_t))^{n-1} [ \mathbf{ L }
(\overline\varepsilon)\mid \mathbf{ L } (\overline\beta)]= \cr
=\textstyle\sum\limits_{\substack{\overline\varepsilon\in
G_{\smash t}^{[k]},\cr
\{\overline\varepsilon\}\supseteq\{\overline\gamma-n+1\}\cr}}
(K_{\smash t}^{[k]}(a_0,a_1,\dots,a_t)) [ \mathbf{ L }
(\overline\alpha)\mid \mathbf{ L } (\overline\varepsilon)]\cdot
(K_{\smash t}^{[k]}(a_0,a_1,\dots,a_t))^{n-1} [ \mathbf{ L }
(\overline\varepsilon))\mid \mathbf{ L } (\overline\beta)]=\cr
\shoveleft{ =\sum\limits_{\substack{\overline\varepsilon\in
G_{\smash t}^{[k]},\cr
\{\overline\varepsilon\}\supseteq\{\overline\gamma-n+1\}\cr}}
\kern-9pt \left(\prod\limits_{i=1}^t\binom{m_{\overline\alpha}(i)}
{m_{\overline\gamma-n}(i)}\!\binom{m_{\overline\varepsilon}(i)}
{m_{\overline\gamma-n+1}(i)}^{-1}\right)\kern-1pt (K_{\smash
t}^{[k]}(a_0,a_1,\dots,a_t)) } \times \cr\hfill\times [ \mathbf{ L }
(\overline\alpha\setminus(\overline\gamma-n))|  \mathbf{ L }
(\overline\varepsilon\setminus(\overline\gamma-n+1))] \times \cr
\times \left(\prod\limits_{i=1}^t\binom{m_{\overline\varepsilon}(i)}
{m_{\overline\gamma-n+1}(i)}\binom{m_{\overline\beta}(i)}
{m_{\overline\gamma}(i)}^{-1}\right) (K_{\smash
t}^{[k]}(a_0,a_1,\dots,a_t))^{n-1} [ \mathbf{ L }
(\overline\varepsilon\setminus(\overline\gamma-n+1))\mid \mathbf{ L
} (\overline\beta\setminus\overline\gamma)]= \cr \medmuskip0mu minus
1mu\thickmuskip0mu minus 1mu
=\left(\prod\limits_{i=1}^t\binom{m_{\overline\alpha}(i)}
{m_{\overline\gamma-n}(i)}\binom{m_{\overline\beta}(i)}
{m_{\overline\gamma}(i)}^{-1}\right) \!
\sum\limits_{\rows{\overline\varepsilon\in G_{\smash t}^{[k]},\cr
\{\overline\varepsilon\}\supseteq\{\overline\gamma-n+1\}\cr}}\hss
(K_{\smash t}^{[k]}(a_0,a_1,\dots,a_t)) [ \mathbf{ L }
(\overline\alpha\setminus(\overline\gamma-n))|  \mathbf{ L }
(\overline\varepsilon\setminus(\overline\gamma-n+1))] \times
\cr\hfill\times (K_{\smash t}^{[k]}(a_0,a_1,\dots,a_t))^{n-1} [
\mathbf{ L }
(\overline\varepsilon\setminus(\overline\gamma-n+1))\mid \mathbf{ L
} (\overline\beta\setminus\overline\gamma)]=\cr
=\left(\prod\limits_{i=1}^t\binom{m_{\overline\alpha}(i)}
{m_{\overline\gamma-n}(i)}\binom{m_{\overline\beta}(i)}
{m_{\overline\gamma}(i)}^{-1}\right) \sum\limits_{\overline\rho\in
G_{\smash t}^{[k]}} (K_{\smash t}^{[k]}(a_0,a_1,\dots,a_t)) [
\mathbf{ L } (\overline\alpha\setminus(\overline\gamma-n))\mid
\mathbf{ L } (\overline\rho)]\times \hfill\cr\noalign{\vskip-6pt}
\hfill\times(K_{\smash t}^{[k]}(a_0,a_1,\dots,a_t))^{n-1} [ \mathbf{
L } (\overline\rho)\mid  \mathbf{ L }
(\overline\beta\setminus\overline\gamma)] = \cr \noalign{\medskip}
\hfill=\left(\prod\limits_{i=1}^t\binom{m_{\overline\alpha}(i)}
{m_{\overline\gamma-n}(i)}\binom{m_{\overline\beta}(i)}
{m_{\overline\gamma}(i)}^{-1}\right) (K_{\smash
t}^{[k]}(a_0,a_1,\dots,a_t))^n [ \mathbf{ L }
(\overline\alpha\setminus(\overline\gamma-n))\mid  \mathbf{ L }
(\overline\beta\setminus\overline\gamma)]. \cr}
$$
\begin{theorem}
Let $k\geqslant 1$ and $t\geqslant 1$, let $a_0$,
$a_1$,~\dots,~$a_t$ be independent variables, let $1\leqslant
n\leqslant t$. Then
\begin{multline}
 \Tr\biggl(\biggl(\frac\partial{\partial a_0}K_{\smash
t}^{[k]}(a_0,a_1, \dots,a_t)\biggr) \cdot(K_{\smash
t}^{[k]}(a_0,a_1,\dots,a_t))^{n-1} \biggr)= \\ = \Tr\biggl( \biggl(
\frac\partial{\partial a_n}K_{\smash t}^{[k]}(a_0,a_1,
\dots,a_t)\biggr)\cdot(K_{\smash t}^{[k]}(a_0,a_1,\dots,a_t))^{n-1}
\biggr).
\end{multline}
\end{theorem}
\begin{proof}
We will be speak, that passage from a vector
$\overline\alpha=(1^{\langle l_1\rangle},2^{\langle
l_2\rangle},\dots ,t^{\langle l_t\rangle})\in G_{\smash t}^{[k]}$
  to a vector $\overline\beta\in G_{\smash t}^{[k]}$ there
exists and be realized with help a vector $(p_1,p_2,\dots,p_{t-1},v)$, if
$$
\overline\beta=\LR(){1^{\left\langle
k-v+\mathop{\Sigma}\limits_{i=1}^{t-1}p_i\right\rangle}},\break
2^{\langle l_1-p_1\rangle},3^{\langle l_2-p_2\rangle},\dots,
t^{\langle l_{t-1}-p_{t-1}\rangle}), \; 0\leqslant p_i\leqslant
l_i, \; 1\leqslant i\leqslant t-1, \;
\sum\limits_{i=1}^{t-1}p_i\leqslant v\leqslant k.
$$
Let $\overline\alpha_1= (1^{\langle l_1\rangle},2^{\langle
l_2\rangle},\dots,t^{\langle l_t\rangle})$, $\overline\alpha_2,
\overline\alpha_3,\dots,\overline\alpha_{n+1}\in G_{\smash t}^{[k]}$
{such that passage from the vector} $\overline\alpha_i$
  {to the vector} $\overline\alpha_{i+1}$ there exist and be
  {realized with help a vector}
$(p_{\smash1}^{(i)},p_{\smash2}^{(i)},\dots,p_{\smash{t-1}}^{(i)},v_i)$,
$1\leqslant i\leqslant n$. Assume that $v_1\geqslant 1$
{and}~$l_n\leqslant k-1$. Let
\begin{align*}
\overline\beta_1& =(1^{\langle l_1\rangle},2^{\langle
l_2\rangle},\dots, (n-1)^{\langle l_{n-1}\rangle},n^{\langle
l_n+1\rangle},(n+1)^{\langle l_{n+1}\rangle},\dots,t^{\langle
l_t\rangle}), \\
\overline\beta_2& = \LR(.{1^{\left\langle
k-v_1+1+\mathop{\Sigma}\limits_{i=1}^{t-1}p_{\smash
i}^{(1)}\right\rangle}},\break2^{\langle
l_1-p_{\smash1}^{(1)}\rangle},3^{\langle
l_2-p_{\smash2}^{(1)}\rangle},\dots, t^{\langle
l_{t-1}-p_{\smash{t-1}}^{(1)}\rangle}).
\end{align*}
Then passage from the vector~$\overline\beta_1$ to~the
vector~$\overline\beta_2$ there exist and realized with help vector
$(p_{\smash1}^{(1)}, \break
p_{\smash2}^{(1)},\dots,p_{\smash{n-1}}^{(1)},
p_{\smash{n}}^{(1)}+1,p_{\smash{n+1}}^{(1)},\dots,p_{\smash{t-1}}^{(1)},
v_1)$ if $1\leqslant n\leqslant t-1$ and
$(p_{\smash1}^{(1)},p_{\smash2}^{(1)},
\dots,p_{\smash{t-1}}^{(1)},v_1-1)$ if $n=t$. Let
$\overline\beta_3$,
$\overline\beta_4$,~\dots,~$\overline\beta_{n+1}\in G_{\smash
t}^{[k]}$ such that passage from the vector~$\overline\beta_i$
to~$\overline\beta_{i+1}$ there exists and be realized with help the
vector
$(p_{\smash1}^{(i)},p_{\smash2}^{(i)},\dots,p_{\smash{t-1}}^{(i)},
v_i)$, $2\leqslant i\leqslant n$. Definition $\overline\beta_3$,
$\overline\beta_4$,~\dots,~$\overline\beta_{n+1}$ is correctly,
since $m_{\overline\beta_2}(i)\geqslant m_{\overline\alpha_2}(i)$
for all~$i$, $1\leqslant i\leqslant t$. Let
$\overline\alpha_i=(1^{\langle l_{\smash1}^{(i)}\rangle}, 2^{\langle
l_{\smash2}^{(i)}\rangle},\dots,t^{\langle l_{\smash
t}^{(i)}\rangle})$, $2\leqslant i\leqslant n+1$. We shall show by
induction on~$n$, $2\leqslant i\leqslant n+1$, that
$$
\overline\beta_i=(1^{\langle l_{\smash1}^{(i)}\rangle},
\dots,(i-2)^{\langle l_{\smash{i-2}}^{(i)}\rangle}, (i-1)^{\langle
l_{\smash{i-1}}^{(i)}+1\rangle}, i^{\langle l_{\smash
i}^{(i)}\rangle},\dots, t^{\langle l_{\smash t}^{(i)}\rangle}).
$$
If $i=2$, then proposition follows from definition
$\overline\alpha_2$ and~$\overline\beta_2$. Let $2\leqslant
i\leqslant n$ and~for~$i$ proposition is proved. Then, by
definition, $\overline\alpha_{i+1} = \LR(){1^{\left\langle
k-v_i+\mathop{\Sigma}\limits _{j=1}^{t-1}p_{\smash
j}^{(i)}\right\rangle}, 2^{\langle
l_{\smash1}^{(i)}-p_{\smash1}^{(i)}\rangle}, 3^{\langle
l_{\smash2}^{(i)}-p_{\smash2}^{(i)}\rangle},\dots, t^{\langle
l_{\smash{t-1}}^{(i)}-p_{\smash{t-1}}^{(i)}\rangle}}$,
$\overline\beta_{i+1} = \LR(.{1^{\left\langle
k-v_i+\mathop{\Sigma}\limits _{j=1}^{t-1}p_{\smash
j}^{(i)}\right\rangle}}\!, 2^{\langle
l_{\smash1}^{(i)}-p_{\smash1}^{(i)}\rangle}, 3^{\langle
l_{\smash2}^{(i)}-p_{\smash2}^{(i)}\rangle},\dots, (i-1)^{\langle
l_{\smash{i-2}}^{(i)}-p_{\smash{i-2}}^{(i)}\rangle}, i^{\langle
l_{\smash{i-1}}^{(i)}+1-p_{\smash{i-1}}^{(i)}\rangle},
(i+1)^{\langle l_{\smash i}^{(i)}-p_{\smash
i}^{(i)}\rangle},\dots,\break t^{\langle
l_{\smash{t-1}}^{(i)}-p_{\smash{t-1}}^{(i)}\rangle})$. We shall show
now that if $v_1>\sum\limits_{j=1}^{t-1}p_{\smash j}^{(1)}$ and
$\overline\alpha_{n+1}=\overline\alpha_1$, then
$$
\displaylines{ \left(\frac\partial{\partial a_0} (K_{\smash
t}^{[k]}(a_0,a_1,\dots,a_t)) [ \mathbf{ L }
(\overline\alpha_1)\mid \mathbf{ L } (\overline\alpha_2)]\right)
\cdot\prod\limits_{i=2}^n (K_{\smash t}^{[k]}(a_0,a_1,\dots,a_t))
[ \mathbf{ L } (\overline\alpha_i)\mid \mathbf{ L }
(\overline\alpha_{i+1})]=
\hfill\cr\hfill=\left(\frac\partial{\partial a_0} (K_{\smash
t}^{[k]}(a_0,a_1,\dots,a_t)) [ \mathbf{ L }
(\overline\beta_1)\mid \mathbf{ L } (\overline\beta_2)]\right)
\cdot\prod\limits_{i=2}^n (K_{\smash t}^{[k]}(a_0,a_1,\dots,a_t))
[ \mathbf{ L } (\overline\beta_i)\mid \mathbf{ L }
(\overline\beta_{i+1})].\cr}
$$
Since $\overline\alpha_{n+1}=\overline\alpha_1$, then
$l_n=\left(k-v_1+\sum\limits_{i=1}^{t-1}p_{\smash i}^{(1)}\right)-
\sum\limits_{j=1}^{n-1}p_{\smash j}^{(j+1)}$ and,~therefore,
$v_1\geqslant 1$ and~$l_n\leqslant k-1$, since
$v_1>\sum\limits_{i=1}^{t-1}p_{\smash i}^{(1)}$. From definition of
matrix
$K_{\smash t}^{[k]}(a_0,a_1,\dots,a_t)$ it follows that 
\begin{multline*}
\left(\frac\partial{\partial a_0} (K_{\smash
t}^{[k]}(a_0,a_1,\dots,a_t)) [ \mathbf{ L }
(\overline\alpha_1)\mid \mathbf{ L } (\overline\alpha_2)]\right)
\cdot\prod\limits_{i=2}^n (K_{\smash t}^{[k]}(a_0,a_1,\dots,a_t))
[ \mathbf{ L } (\overline\alpha_i)\mid \mathbf{ L }
(\overline\alpha_{i+1})]= \hfill\cr \shoveleft{ =\Biggl(
\frac\partial{\partial a_0} \Binom{k\mathstrut}
{\!v_1-\mathop{\Sigma}\limits_{j=1}^{t-1}p_{\smash j}^{(1)}\!}\!
\left(\displaystyle \prod\limits_{j=1}^{t-1}\binom{l_j}{p_{\smash
j}^{(1)}}\right) a_{\smash
0}^{v_1-\mathop{\Sigma}\limits_{j=1}^{t-1} p_{\smash
j}^{(1)}}a_{\smash1}^{p_{\smash1}^{(1)}}a_{\smash2}^
{p_{\smash2}^{(1)}}\!\dots
a_{\smash{t-1}}^{p_{\smash{t-1}}^{(1)}} \times } \\ \shoveright{
\times \displaystyle\sum\limits_{d=0}^{\min(l_{\smash
t}^{(i)},k-v_1)}\!\!\! \!(d{+}v_1)!\!\binom
k{d{+}v_1}\!\binom{l_t}d\!a_{\smash t}^d\mkern-2mu \Biggr)
\prod\limits_{i=2}^n\Binom{k\mathstrut}
{v_i-\mathop{\Sigma}\limits_{j=1}^{t-1}p_{\smash j}^{(i)}}
\left(\displaystyle \prod\limits_{j=1}^{t-1}\binom{l_{\smash
j}^{(i)}}{p_{\smash j}^{(i)}}\right) \times } \\ \times a_{\smash
0}^{v_i-\mathop{\Sigma}\limits_{j=1}^{t-1} p_{\smash
j}^{(i)}}a_{\smash1}^{p_{\smash1}^{(i)}}a_{\smash2}^
{p_{\smash2}^{(i)}}\!\dots
a_{\smash{t-1}}^{p_{\smash{t-1}}^{(i)}} \displaystyle
\sum\limits_{d=0}^{\min(l_{\smash t}^{(i)},k-v_i)}
(d{+}v_i)!\binom k{d{+}v_i}\!\binom{l_{\smash t}^{(i)}}da_{\smash
t}^d.
\end{multline*}
As we proved above,
$\overline\beta_i=(1^{\langle l_{\smash1}^{(i)}\rangle},
\dots,(i-2)^{\langle l_{\smash{i-2}}^{(i)}\rangle}, (i-1)^{\langle
l_{\smash{i-1}}^{(i)}+1\rangle}, i^{\langle l_{\smash
i}^{(i)}\rangle},\dots, t^{\langle l_{\smash t}^{(i)}\rangle})$,
$2\leqslant i\leqslant n+1$.\break Therefore,
$\overline\beta_{n+1}=\overline\beta_1$~and
$$
\displaylines{ \left(\frac\partial{\partial a_n} (K_{\smash
t}^{[k]}(a_0,a_1,\dots,a_t)) [ \mathbf{ L }
(\overline\beta_1)\mid \mathbf{ L } (\overline\beta_2)]\right)
\cdot\prod\limits_{i=2}^n (K_{\smash t}^{[k]}(a_0,a_1,\dots,a_t))
[ \mathbf{ L } (\overline\beta_i)\mid \mathbf{ L }
(\overline\beta_{i+1})]= \hfill\cr =\LR(.{\frac\partial{\partial
a_n}\LR(.{ \Binom{k\mathstrut} {v_1-1-\mathop
{\Sigma}\limits_{j=1}^{t-1}p_{\smash j}^{(1)}}
\binom{l_1}{p_{\smash{1}}^{(1)}}
\binom{l_2}{p_{\smash{2}}^{(1)}}\dots
\binom{l_{n-1}}{p_{\smash{n-1}}^{(1)}}
\binom{l_n+1}{p_{\smash{n}}^{(1)}+1}
\binom{l_{n+1}}{p_{\smash{n+1}}^{(1)}}\dots
\binom{l_{t-1}}{p_{\smash{t-1}}^{(1)}}
a_{\smash0}^{v_1-1-\mathop{\Sigma}\limits_{j=1}^{t-1}p_{\smash
j}^{(1)}} }} \times \hfill\cr\times\left.\left.
a_{\smash{1}}^{p_{\smash{1}}^{(1)}}
a_{\smash{2}}^{p_{\smash{2}}^{(1)}}\dots
a_{\smash{n-1}}^{p_{\smash{n-1}}^{(1)}}
a_{\smash{n}}^{p_{\smash{n}}^{(1)}+1}
a_{\smash{n+1}}^{p_{\smash{n+1}}^{(1)}}\dots
a_{\smash{t-1}}^{p_{\smash{t-1}}^{(1)}}
\sum\limits_{d=0}^{\min(l_t,k-v_1)}(d+v_1)!\binom
k{d+v_1}\binom{l_t}d a_t^d\right)\right) \times \cr
\times\prod\limits_{i=2}^n \Binom{k\mathstrut} {v_i-\mathop
{\Sigma}\limits_{j=1}^{t-1}p_{\smash j}^{(i)}}
\binom{l_{\smash1}^{(i)}}{p_{\smash{1}}^{(i)}}
\binom{l_{\smash2}^{(i)}}{p_{\smash{2}}^{(i)}}\dots
\binom{l_{\smash{i-2}}^{(i)}}{p_{\smash{i-2}}^{(i)}}
\binom{l_{\smash{i-1}}^{(i)}+1}{p_{\smash{i-1}}^{(i)}+1}
\binom{l_{\smash{i}}^{(i)}}{p_{\smash{i}}^{(i)}}\dots
\binom{l_{\smash{t-1}}^{(i)}}{p_{\smash{t-1}}^{(i)}} \times
\cr\hfill\times
a_{\smash0}^{v_i-\mathop{\Sigma}\limits_{j=1}^{t-1}p_{\smash
j}^{(i)}} a_{\smash{1}}^{p_{\smash{1}}^{(i)}}
a_{\smash{2}}^{p_{\smash{2}}^{(i)}}\dots
a_{\smash{t-1}}^{p_{\smash{t-1}}^{(i)}}
\sum\limits_{d=0}^{\min(l_{\smash t}^{(i)}, k-v_i)}(d+v_i)!\binom
k{d+v_i}\binom {l_{\smash t}^{(i)}}da_t^d,\cr}
$$
if $1\leqslant n\leqslant t-1$ and
$$
\displaylines{ \left(\frac\partial{\partial a_t} (K_{\smash
t}^{[k]}(a_0,a_1,\dots,a_t)) [ \mathbf{ L } (\overline\beta_1)\mid
\mathbf{ L } (\overline\beta_2)]\right) \cdot\prod\limits_{i=2}^t
(K_{\smash t}^{[k]}(a_0,a_1,\dots,a_t)) [ \mathbf{ L }
(\overline\beta_i)\mid \mathbf{ L } (\overline\beta_{i+1})]=
\hfill\cr =\LR(.{\frac\partial{\partial a_t}\LR(.{
\Binom{k\mathstrut} {v_1-1-\mathop
{\Sigma}\limits_{j=1}^{t-1}p_{\smash j}^{(1)}}
\left(\displaystyle\prod\limits_{j=1}^{t-1}\binom{l_j}
{p_{\smash{j}}^{(1)}}\right)
a_{\smash0}^{v_1-1-\mathop{\Sigma}\limits_{j=1}^{t-1}p_{\smash
j}^{(1)}} a_{\smash{1}}^{p_{\smash{1}}^{(1)}}
a_{\smash{2}}^{p_{\smash{2}}^{(1)}}\dots
a_{\smash{t-1}}^{p_{\smash{t-1}}^{(1)}} }} \times
\hfill\cr\times\left.\left.
\sum\limits_{d=0}^{\min(l_t+1,k-v_1+1)}(d+v_1-1)!\binom
k{d+v_1-1}\binom{l_t+1}d a_t^d\right)\right) \prod\limits_{i=2}^t
\Binom{k\mathstrut} {v_i -\mathop
{\Sigma}\limits_{j=1}^{t-1}p_{\smash j}^{(i)}} \times \cr
\times \binom{l_{\smash1}^{(i)}}{p_{\smash{1}}^{(i)}}\dots
\binom{l_{\smash{i-2}}^{(i)}}{p_{\smash{i-2}}^{(i)}}
\binom{l_{\smash{i-1}}^{(i)}+1}{p_{\smash{i-1}}^{(i)}}
\binom{l_{\smash{i}}^{(i)}}{p_{\smash{i}}^{(i)}}\dots
\binom{l_{\smash{t-1}}^{(i)}}{p_{\smash{t-1}}^{(i)}}
a_{\smash0}^{v_i-\mathop{\Sigma}\limits_{j=1}^{t-1}p_{\smash
j}^{(i)}} a_{\smash{1}}^{p_{\smash{1}}^{(i)}}
a_{\smash{2}}^{p_{\smash{2}}^{(i)}}\dots
a_{\smash{t-1}}^{p_{\smash{t-1}}^{(1)}} \times \cr\hfill \times
\sum\limits_{d=0}^{\min(l_{\smash t}^{(i)},k-v_i)}(d+v_i)!\binom
k{d+v_i}\binom {l_{\smash t}^{(i)}}da_t^d.\cr}
$$
Let $s=v_1-\sum\limits_{j=1}^{t-1}p_{\smash j}^{(1)}$. Then
$$
\displaylines{ \left(\frac\partial{\partial a_0} (K_{\smash
t}^{[k]}(a_0,a_1,\dots,a_t)) [ \mathbf{ L } (\overline\alpha_1)\mid
\mathbf{ L } (\overline\alpha_2)]\right)
\prod\limits_{i=2}^n(K_{\smash t}^{[k]}(a_0,a_1,\dots,a_t)) [
\mathbf{ L } (\overline\alpha_i)\mid \mathbf{ L }
(\overline\alpha_{i+1})]\hfill\cr =s\binom
ks\left(\prod\limits_{j=1}^{t-1}\binom{l_j}{p_{\smash j}^{(1)}}
a_{\smash j}^{p_{\smash j}^{(1)}}\right)a_0^{s-1}\left(\sum\limits
_{d=0}^{\min(l_t,k-v_1)}(d+v_1)!\binom
k{d+v_1}\binom{l_t}da_t^d\right) \times \hfill\cr\hfill\times
\prod\limits_{i=2}^n\Binom{k\mathstrut}{v_i-\mathop{\Sigma}\limits_
{j=1}^{t-1}p_{\smash j}^{(i)}}\left(\prod\limits_{j=1}^{t-1}\binom
{l_{\smash j}^{(i)}}{p_{\smash j}^{(i)}}a_{\smash j}^{p_{\smash j}
^{(i)}}\right)a_{\smash0}^{v_i-\mathop{\Sigma}\limits_{j=1}^{t-1}
p_{\smash j}^{(i)}}\sum\limits_{d=0}^{\min(l_{\smash
t}^{(i)},k-v_i)} (d+v_i)!\binom k{d+v_i}\binom{l_{\smash
t}^{(i)}}da_t^d,\cr
\left(\frac\partial{\partial a_n} (K_{\smash
t}^{[k]}(a_0,a_1,\dots,a_t)) [ \mathbf{ L } (\overline\beta_1)\mid
\mathbf{ L } (\overline\beta_2)]\right)
\prod\limits_{i=2}^n(K_{\smash t}^{[k]}(a_0,a_1,\dots,a_t)) [
\mathbf{ L } (\overline\beta_i)\mid \mathbf{ L }
(\overline\beta_{i+1})]= \hfill\cr
=(p_n^{(1)}+1)\cdot\frac{l_n+1}{p_n^{(1)}+1}\binom k{s-1}\left(\prod
\limits_{j=1}^{t-1}\binom{l_j}{p_{\smash j}^{(1)}}a_{\smash j}^{p_{
\smash j}^{(1)}}\right)a_0^{s-1}\left(\sum\limits_{d=0}^{\min(l_t,
k-v_1)}(d+v_1)!\binom k{d+v_1}\binom{l_t}da_t^d\right)\times
\hfill\cr
\times\left(\prod\limits_{i=2}^n\binom{l_{\smash{i-1}}^{(i)}+1}
{p_{\smash{i-1}}^{(i)}}\right)\prod\limits_{i=2}^n
\Binom{k\mathstrut}{v_i-\mathop{\Sigma}\limits_{j=1} ^{t-1}p_{\smash
j}^{(i)}}\LR(){\displaystyle \prod\limits_{j\in N_{t-1}\setminus\{
i-1\}}\binom{l_{\smash j}^{(i)}}{p_{\smash j}^{(i)}}}
a_{\smash0}^{v_i-\mathop{\Sigma}\limits_{j=1}^{t-1}p_{\smash
j}^{(i)}}
a_{\smash1}^{p_{\smash1}^{(i)}}a_{\smash2}^{p_{\smash2}^{(i)}}\dots
a_{\smash{t-1}}^{p_{\smash{t-1}}^{(i)}} \times \cr\hfill\times
\sum\limits_{d=0}^{\min(l_{\smash t}^{(i)},k-v_i)}(d+v_i)!\binom
k{d+v_i}\binom{l_{\smash t}^{(i)}}da_t^d\cr}
$$
if $1\leqslant n\leqslant t-1$ and
$$
\displaylines{ \left(\frac\partial{\partial a_t} (K_{\smash
t}^{[k]}(a_0,a_1,\dots,a_t)) [ \mathbf{ L }
(\overline\beta_1)\mid \mathbf{ L } (\overline\beta_2)]\right)
\prod\limits_{i=2}^t(K_{\smash t}^{[k]}(a_0,a_1,\dots,a_t)) [
\mathbf{ L } (\overline\beta_i)\mid \mathbf{ L }
(\overline\beta_{i+1})]= \hfill\cr
=\binom k{s-1}a_0^{s-1}\left(\prod\limits_{j=1}^{t-1}\binom{l_j}
{p_{\smash j}^{(1)}}a_{\smash j}^{p_{\smash j}^{(1)}}\right)
\left(\sum\limits_{d=1}^{\min(l_t+1,k-v_1+1)}(d+v_1-1)!\binom k
{d+v_1-1}\binom{l_t+1}dda_t^{d-1}\right)\times \hfill\cr
\times\left(\prod\limits
_{i=2}^t\binom{l_{\smash{i-1}}^{(i)}+1}{p_{\smash{i-1}}^{(i)}}\right)
\prod\limits_{i=2}^t\Binom{k\mathstrut}{v_i-\mathop{\Sigma}\limits_{j=1}
^{t-1}p_{\smash j}^{(i)}}\LR(){\displaystyle \prod\limits_{j\in
N_{t-1}\setminus\{ i-1\}}\binom{l_{\smash j}^{(i)}}{p_{\smash
j}^{(i)}}}
a_{\smash0}^{v_i-\mathop{\Sigma}\limits_{j=1}^{t-1}p_{\smash
j}^{(i)}} \left(\prod\limits_{j=1}^{t-1}a_{\smash j}^{p_{\smash
j}^{(i)}}\right)\times \cr\hfill\times
\sum\limits_{d=0}^{\min(l_{\smash t}^{(i)},k-v_i)}(d+v_i)! \binom k
{d+v_i} \binom{l_{\smash t}^{(i)}}da_t^d = \cr = 
(l_t+1)\binom k{s-1}a_0^{s-1}\left(\prod
\limits_{j=1}^{t-1}\binom{l_j}{p_{\smash j}^{(1)}}a_{\smash j}^{p_{
\smash j}^{(1)}}\right)\left(\sum\limits_{d=0}^{\min(l_t,
k-v_1)}(d+v_1)!\binom k{d+v_1}\binom{l_t}da_t^d\right)\times
\hfill\cr \hfill \times
\left(\prod\limits_{i=2}^n\binom{l_{\smash{i-1}}^{(i)}+1}{p_{\smash
{i-1}}^{(i)}}\right)\prod
\limits_{i=2}^n\Binom{k\mathstrut}{v_i-\mathop{\Sigma}\limits_{j=1}
^{t-1}p_{\smash j}^{(i)}}\LR(){\displaystyle \prod\limits_{j\in
N_{t-1}\setminus\{ i-1\}}\binom{l_{\smash j}^{(i)}}{p_{\smash
j}^{(i)}}}
a_{\smash0}^{v_i-\mathop{\Sigma}\limits_{j=1}^{t-1}p_{\smash
j}^{(i)}} \left(\prod\limits_{j=1}^{t-1}a_{\smash j}^{p_{\smash
j}^{(i)}}\right) \times \hfill\cr \hfill \times
\sum\limits_{d=0}^{\min(l_{\smash t}^{(i)},k-v_i)}(d+v_i)!\binom
k{d+v_i}\binom{l_{\smash t}^{(i)}}da_t^d.\cr}
$$
Therefore, for all $n$, $1\leqslant n\leqslant t$, this equality
$$
\displaylines{\left(\frac\partial{\partial a_n} (K_{\smash
t}^{[k]}(a_0,a_1,\dots,a_t)) [ \mathbf{ L }
(\overline\beta_1)\mid \mathbf{ L } (\overline\beta_2)]\right)
\prod\limits_{i=2}^n(K_{\smash t}^{[k]}(a_0,a_1,\dots,a_t)) [
\mathbf{ L } (\overline\beta_i)\mid \mathbf{ L }
(\overline\beta_{i+1})] = \hfill\cr =(l_n+1)\binom
k{s-1}a_0^{s-1}\left(\prod\limits_{j=1}^{t-1}\binom{l_j}
{p_{\smash j}^{(1)}}a_{\smash j}^{p_{\smash
j}^{(1)}}\right)\left(\sum
\limits_{d=0}^{\min(l_t,k-v_1)}(d+v_1)!\binom
k{d+v_1}\binom{l_t}da_t^d \right) \times \hfill\cr\hfill\times
\left(\prod\limits_{i=2}^n\binom{l_{\smash{i-1}}^{(i)}+1}
{p_{\smash{i-1}}^{(i)}}\right)
\prod\limits_{i=2}^n\Binom{k\mathstrut}
{v_i-\mathop{\Sigma}\limits_{j=1}^{t-1}p_{\smash
j}^{(i)}}\LR(){\prod \limits_{j\in
N_{t-1}\setminus\{i-1\}}\binom{l_{\smash j}^{(i)}} {p_{\smash
j}^{(i)}}}
a_{\smash0}^{v_i-\mathop{\Sigma}\limits_{j=1}^{t-1}p_{\smash
j}^{(i)}} \left(\prod\limits_{j=1}^{t-1}a_{\smash j}^{p_{\smash
j}^{(i)}}\right) \times \hfill\cr\hfill\times
\sum\limits_{d=0}^{\min(l_{\smash t}^{(i)},k-v_i)}(d+v_i)!\binom
k {d+v_i}\binom{l_{\smash t}^{(i)}}da_t^d.\cr}
$$
is correct. From the equality $l_{\smash
j}^{(i+1)}=l_{\smash{j-1}}^{(i)}-p_{\smash{j-1}}^{(i)}$, $2\leqslant
j\leqslant t$, $1\leqslant i\leqslant n-1$, it follows that
$l_{\smash i}^{(i+1)}=l_{\smash{i-1}}^{(i)}-p_{\smash{i-1}}^{(i)}$,
$2\leqslant i\leqslant n-1$. Since $l_{\smash1}^{(2)}=k-s$, then
induction on~$i$ give that
$l_{\smash{i-1}}^{(i)}=k-s-\sum\limits_{j=1}^{i-2}p_{\smash
j}^{(j+1)}$, $2\leqslant i\leqslant n$. Therefore
$$
\displaylines{ s\binom
ks\prod\limits_{i=2}^n\binom{l_{\smash{i-1}}^{(i)}}
{p_{\smash{i-1}}^{(i)}}=k\binom{k-1}{s-1}\prod\limits_{i=2}^n\binom
{l_{\smash{i-1}}^{(i)}}{p_{\smash{i-1}}^{(i)}}=(k-s+1)\binom
k{s-1}
\prod\limits_{i=2}^n\binom{l_{\smash{i-1}}^{(i)}}{p_{\smash{i-1}}^{(i)}}=
\hfill\cr =(k-s+1)\binom
k{s-1}\prod\limits_{i=2}^n\Binom{k-s-\mathop{\Sigma}\limits
_{j=1}^{i-2}p_{\smash j}^{(j+1)}}{p_{\smash{i-1}}^{(i)}} = \hfill \cr \hfill =
(k-s+1)\binom k{s-1}\cdot\frac{(k-s)!}
{\textstyle\left(k-s-\sum\limits_{i=2}^np_{\smash{i-1}}^{(i)}\right)!
\prod\limits_{i=2}^n p_{\smash{i-1}}^{(i)}!} =\cr\hfill =\binom
k{s-1}\cdot\frac{(k-s+1)!}{\textstyle\left(k-s-\sum\limits
_{i=2}^np_{\smash{i-1}}^{(i)}\right)!\prod\limits_{i=2}^n
p_{\smash{i-1}}^{(i)}!} = \binom
k{s-1}\cdot\frac{(k-s+1)!}{\textstyle l_n!\prod\limits_{i=2}^n
p_{\smash{i-1}}^{(i)}!}, \cr\noalign{\bigskip} (l_n+1)\binom
k{s-1}\prod\limits_{i=2}^n\binom{l_{\smash{i-1}}^{(i)}+1}
{p_{\smash{i-1}}^{(i)}} = (l_n+1)\binom
k{s-1}\prod\limits_{i=2}^n\Binom{k-s+1-\mathop{\Sigma}
\limits_{j=1}^{i-2}p_{\smash j}^{(j+1)}}{p_{\smash{i-1}}^{(i)}}
=\hfill\cr\hfill = (l_n+1)\binom k{s-1} \times
\frac{(k-s+1)!}{\left(k-s+1-\sum\limits_{i=2}^n
p_{\smash{i-1}}^{(i)}\right)!\prod\limits_{i=2}^np_{\smash{i-1}}^{(i)}!}=(l_n+1)\binom
k{s-1}\cdot\frac{(k-s+1)!}{(l_n+1)!\prod\limits_{i=2}^n
p_{\smash{i-1}}^{(i)}!} = \hfill\cr\hfill=\binom
k{s-1}\cdot\frac{(k-s+1)!}{l_n!\prod\limits_{i=2}^n
p_{\smash{i-1}}^{(i)}!},\cr}
$$
$1\leqslant n\leqslant t$. Therefore, if
$\overline\alpha_{n+1}=\overline\alpha_1$, then
$\overline\beta_{n+1}=\overline\beta_1$~and
$$
\displaylines{\left(\frac\partial{\partial a_0} (K_{\smash
t}^{[k]}(a_0,a_1,\dots,a_t)) [ \mathbf{ L }
(\overline\alpha_1)\mid \mathbf{ L } (\overline\alpha_2)]\right)
\prod\limits_{i=2}^n(K_{\smash t}^{[k]}(a_0,a_1,\dots,a_t)) [
\mathbf{ L } (\overline\alpha_i)\mid \mathbf{ L }
(\overline\alpha_{i+1})]= \hfill\cr
\hfill=\left(\frac\partial{\partial a_n} (K_{\smash
t}^{[k]}(a_0,a_1,\dots,a_t)) [ \mathbf{ L }
(\overline\beta_1)\mid \mathbf{ L } (\overline\beta_2)]\right)
\prod\limits_{i=2}^n(K_{\smash t}^{[k]}(a_0,a_1,\dots,a_t)) [
\mathbf{ L } (\overline\beta_i)\mid \mathbf{ L }
(\overline\beta_{i+1})].\cr}
$$
  And invercely, let $\overline\beta=(1^{\langle l_1\rangle},
2^{\langle l_2\rangle},\dots,t^{\langle l_t\rangle})$,
$\overline\beta_2$,~\dots,~$\overline\beta_{n+1}\in G_{\smash t}^{[k]}$
  {be such that passage from the vector}
$\overline\beta_i$ to the~vector $\overline\beta_{i+1}$ be realized
with help the vector $(p_{\smash1}^{\langle
i\rangle},p_{\smash2}^{\langle i\rangle},\dots
p_{\smash{t-1}}^{\langle i\rangle},v_i)$, $1\leqslant i\leqslant n$.
Assume that $v_1\leqslant k- 1$ and~$p_n^{(1)}\geqslant 1$,
$l_n\geqslant 1$. Let
\begin{align*}
\overline\alpha_1 &=(1^{\langle l_1\rangle},2^{\langle
l_2\rangle},\dots, (n-1)^{\langle l_{n-1}\rangle},n^{\langle
l_n-1\rangle},(n+1)^{\langle l_{n+1}\rangle},\dots,\break
t^{\langle l_t\rangle}), \\
\overline\alpha_2 &= \LR(){1^{\left\langle
k-v_1-1+\mathop{\Sigma}\limits _{j=1}^{t-1}p_{\smash
j}^{(1)}\right\rangle}, 2^{\langle l_1-p_{\smash1}^{(1)}\rangle},
3^{\langle l_2-p_{\smash2}^{(1)}\rangle},\dots, t^{\langle
l_{t-1}-p_{\smash{t-1}}^{(1)}\rangle}}.
\end{align*}
Then passage from the vector $\overline\alpha_1$
to~$\overline\alpha_2$ be realized with help the vector
$$
(p_{\smash{1}}^{(1)},p_{\smash{2}}^{(1)},\dots,
p_{\smash{n-1}}^{(1)},p_{\smash{n}}^{(1)}-1,p_{\smash{n+1}}^{(1)},\dots,
p_{\smash{t-1}}^{(1)},v_1) \mbox{ if } 1\leqslant n\leqslant t-1
$$
and~$(p_{\smash{1}}^{(1)},
p_{\smash{2}}^{(1)},\dots,p_{\smash{t-1}}^{(1)}, v_1+1)$ if
$n=t$.  Let $\overline\alpha_3$,
$\overline\alpha_4$,~\dots,~$\overline\alpha_{n+1}\in G_{\smash
t}^{[k]}$, such that passage from~$\overline\alpha_i$
to~$\overline\alpha_{i+1}$ be realized with help the vector
$(p_{\smash{1}}^{(i)},p_{\smash{2}}^{(i)},\dots,p_{\smash{t-1}}^{(i)},v_i)$,
$2\leqslant i\leqslant n$. If
$\overline\beta_{n+1}=\overline\beta_1$, then definition of
$\overline\alpha_3$,
$\overline\alpha_4$,~\dots,~$\overline\alpha_{n+1}$ is correctly.
Really then $l_n=\Bigl(k-v_1+\sum\limits_{j=1}^{t-1}p_{\smash
j}^{(1)} \Bigr)-\sum\limits_{j=1}^{n-1}p_{\smash j}^{(j+1)}$. Let
$\overline\beta_i=(1^{\langle l_{\smash1}^{(i)}\rangle},
2^{\langle l_{\smash2}^{(i)}\rangle},\dots,t^{\langle l_{\smash
t}^{(i)}\rangle})$, $2\leqslant i\leqslant n+1$. Then
$$
l_{\smash{i-1}}^{(i)}-1=\left(k-v_1-1+\sum\limits_{j=1}^{t-1}p_{\smash
j}^{(1)}\right)-\sum\limits_{j=1}^{i-2}p_{\smash
j}^{(j+1)}\geqslant l_n-1\geqslant 0
$$
and~$\overline\alpha_i=(1^{\langle l_{\smash1}^{(i)}\rangle},
\ldots, (i-2)^{\langle l_{\smash{i-2}}^{(i)}\rangle}, (i-1)^{\langle
l_{\smash{i-1}}^{(i)}-1\rangle}, i^{\langle l_{\smash
i}^{(i)}\rangle},\dots,t^{\langle l_{\smash t}^{(i)}\rangle})$.
Therefore, if $\overline\beta_{n+1}=\overline\beta_1$, then
$\overline\alpha_{n+1}=\overline\alpha_1$ and,~repeating stated
above computations we obtain that
$$
\displaylines{\left(\frac\partial{\partial a_0} (K_{\smash
t}^{[k]}(a_0,a_1,\dots,a_t)) [ \mathbf{ L }
(\overline\alpha_1)\mid \mathbf{ L } (\overline\alpha_2)]\right)
\prod\limits_{i=2}^n(K_{\smash t}^{[k]}(a_0,a_1,\dots,a_t)) [
\mathbf{ L } (\overline\alpha_i)\mid \mathbf{ L }
(\overline\alpha_{i+1})]= \hfill\cr
\hfill=\left(\frac\partial{\partial a_n} (K_{\smash
t}^{[k]}(a_0,a_1,\dots,a_t)) [ \mathbf{ L }
(\overline\beta_1)\mid \mathbf{ L } (\overline\beta_2)]\right)
\prod\limits_{i=2}^n(K_{\smash t}^{[k]}(a_0,a_1,\dots,a_t)) [
\mathbf{ L } (\overline\beta_i)\mid \mathbf{ L }
(\overline\beta_{i+1})].\cr}
$$
Since the maps
$$
(\overline\alpha_1,\overline\alpha_2,\dots,
\overline\alpha_n,\overline\alpha_1) \to
(\overline\beta_1,\overline\beta_2,
\dots,\overline\beta_n,\overline\beta_1) \text{ and }
(\overline\beta_1,\overline
\beta_2,\dots,\overline\beta_n,\overline\beta_1) \to
(\overline\alpha_1,
\overline\alpha_2,\dots,\break\overline\alpha_n,\overline\alpha_1)
$$
are injective, then theorem 10 is proved.
\end{proof}

\begin{theorem}
Let $k \geqslant 1,$ $t \geqslant 1,$ $1 \leqslant n \leqslant t,$
$0 \leqslant r \leqslant kt-1,$ let $a_0, a_1, \dots, a_t$ be
independent variables. Then
\begin{multline}
\Tr\Bigl( \Bigl( \frac{\partial}{\partial a_0} \Pi_{r,t}^{[k]}(a_0,
a_1, \dots, a_t) \Bigr) \cdot \Bigl( \Pi_{r,t}^{[k]}(a_0,a_1, \dots,
a_t) \Bigr)^{n-1} \Bigr) = \\= \Tr \Bigl( \Bigl(
\frac{\partial}{\partial a_n} \Pi_{r+1,t}^{[k]}(a_0, a_1, \dots,
a_t) \Bigr) \cdot \Bigl( \Pi_{r+1,t}^{[k]}(a_0,a_1, \dots, a_t)
\Bigr)^{n-1} \Bigr)
\end{multline}
\end{theorem}
\begin{theorem}
Let $k \geqslant 1$ and $t \geqslant 2$, let $a_0, a_1, \dots, a_t$
be elements in a commutative ring with unity, let
$\overline{\alpha}, \overline{\beta} \in G_{r,t}^{[k]},$ let $1
\leqslant n \leqslant t-1$. Assume that $\max \{\overline{\beta} \}
\geqslant n+1,$ $\overline{\gamma} \in G_{l,t}^{[k]},$ $\{
\overline{\gamma} \} \subseteq \{ \overline{\beta} \},$
$\min\{\overline{\gamma}\} \geqslant n+1$ and $\{ \overline{\alpha}
\} \nsupseteq \{ \overline{\gamma}-n \}.$ Then
$$
\Bigl( \Pi_{r,t}^{[k]} (a_0,a_1,\dots, a_t) \Bigr)^n
[\mathbf{L}(\overline{\alpha}) \mid \mathbf{L}(\overline{\beta})]= 0
$$
\end{theorem}
\begin{theorem}
Let $k \geqslant 1$ and $t \geqslant 2$, let $a_0, a_1, \dots, a_t$
be elements in a commutative ring with unity, let
$\overline{\alpha}, \overline{\beta} \in G_{r,t}^{[k]},$ let $1
\leqslant n \leqslant t-1$. Assume that $\max \{\overline{\beta} \}
\geqslant n+1,$ $\overline{\gamma} \in G_{l,t}^{[k]},$ $\{
\overline{\gamma} \} \subseteq \{ \overline{\beta} \},$
$\min\{\overline{\gamma}\} \geqslant n+1$, and $\{ \overline{\alpha}
\} \supseteq \{ \overline{\gamma}-n \}.$ Then
\begin{multline}
\Bigl( \Pi_{r,t}^{[k]} (a_0, a_1, \dots, a_t) \Bigr)^n
[\mathbf{L}(\overline{\alpha}) \mid \mathbf{L}(\overline{\beta})] =
\\ = \biggl( \prod_{i=1}^t
\binom{m_{\overline{\alpha}}(i)}{m_{\overline{\gamma}-n}(i)}
\binom{m_{\overline{\beta}}(i)}{m_{\overline{\gamma}}(i)}^{-1}
\biggr) \Bigl( \Pi_{r-l,t}^{[k]}(a_0,a_1, \dots, a_t) \Bigr)^n
[\mathbf{L}(\overline{\alpha}\setminus(\overline{\gamma}-n)) \mid
\mathbf{L}(\overline{\beta}\setminus \overline{\gamma})].
\end{multline}
\end{theorem}
\begin{theorem}
Let $k \geqslant 1$ and $t \geqslant 1,$ let $1 \leqslant n
\leqslant t,$ let $a_0,a_1, \dots, a_t$ be independent variables.
Then
\begin{multline}
\sum_{r=0}^{kt}\Tr\Bigl( \Bigl( \frac{\partial}{\partial a_0}
\Pi_{r,t}^{[k]}(a_0, a_1, \dots, a_t) \Bigr) \cdot \Bigl(
\Pi_{r,t}^{[k]}(a_0,a_1, \dots, a_t) \Bigr)^{n-1} \Bigr) = \\=
\sum_{r=0}^{kt} \Tr \Bigl( \Bigl( \frac{\partial}{\partial a_n}
\Pi_{r,t}^{[k]}(a_0, a_1, \dots, a_t) \Bigr) \cdot \Bigl(
\Pi_{r,t}^{[k]}(a_0,a_1, \dots, a_t) \Bigr)^{n-1} \Bigr)
\end{multline}
\end{theorem}
\begin{proof}
Since $\Pi_{0,t}^{[k]}(a_0,a_1, \dots, a_t)\! = k! a_0^k,$
$\Pi_{kt,t}^{[k]}(a_0,a_1, \dots, a_t)\! = k! a_t^k,$ then
$\frac{\partial}{\partial a_0} \Pi_{kt,t}^{[k]}(a_0,a_1, \dots, a_t)
= \break = 0,$ $\frac{\partial}{\partial a_n} \Pi_{0,t}^{[k]}
(a_0,a_1, \dots, a_t) = 0$ and from the theorem 11 it follows that
\begin{multline*}
\sum_{r=0}^{kt}\Tr\Bigl( \Bigl( \frac{\partial}{\partial a_0}
\Pi_{r,t}^{[k]}(a_0, a_1, \dots, a_t) \Bigr) \cdot \Bigl(
\Pi_{r,t}^{[k]}(a_0,a_1, \dots, a_t) \Bigr)^{n-1} \Bigr) = \\=
\sum_{r=0}^{kt-1} \Tr \Bigl( \Bigl( \frac{\partial}{\partial a_0}
\Pi_{r,t}^{[k]}(a_0, a_1, \dots, a_t) \Bigr) \cdot \Bigl(
\Pi_{r,t}^{[k]}(a_0,a_1, \dots, a_t) \Bigr)^{n-1} \Bigr) = \\=
\sum_{r=0}^{kt-1} \Tr \Bigl( \Bigl( \frac{\partial}{\partial a_n}
\Pi_{r+1,t}^{[k]}(a_0, a_1, \dots, a_t) \Bigr) \cdot \Bigl(
\Pi_{r+1,t}^{[k]}(a_0,a_1, \dots, a_t) \Bigr)^{n-1} \Bigr) = \\ =
\sum_{r=1}^{kt} \Tr \Bigl( \Bigl( \frac{\partial}{\partial a_n}
\Pi_{r,t}^{[k]}(a_0, a_1, \dots, a_t) \Bigr) \cdot \Bigl(
\Pi_{r,t}^{[k]}(a_0,a_1, \dots, a_t) \Bigr)^{n-1} \Bigr) = \\ =
\sum_{r=0}^{kt} \Tr \Bigl( \Bigl( \frac{\partial}{\partial a_n}
\Pi_{r,t}^{[k]}(a_0, a_1, \dots, a_t) \Bigr) \cdot \Bigl(
\Pi_{r,t}^{[k]}(a_0,a_1, \dots, a_t) \Bigr)^{n-1} \Bigr)
\end{multline*}
\end{proof}
\begin{theorem}
Let $k \geqslant 1$ and $n \geqslant 1,$ let $0 \leqslant r
\leqslant kt,$ let $a_0,a_1, \dots, a_t$ be independent variables.
Then
$$
\Tr \Bigl( \Bigl( \Pi_{r,t}^{[k]}(a_0, a_1x, \dots, a_t x^t)
\Bigr)^n \Bigr) = x^{rn} \Tr \Bigl( \Bigl( \Pi_{r,t}^{[k]}(a_0,a_1,
\dots, a_t) \Bigr)^n \Bigr) \eqno(7)
$$
\end{theorem}
\begin{proof}
Let $\overline{\alpha_i} = (1^{\langle l_1^{(i)} \rangle},
2^{\langle l_2^{(i)} \rangle}, \dots, t^{\langle l_t^{(i)} \rangle})
\! \in \! G_{r,t}^{[k]},$ $1 \leqslant\! i \!\leqslant \! n+1.$
Assume that $\overline{\alpha_1} \to \overline{\alpha_2} \to \dots
\to \overline{\alpha_{n+1}}$ is a closed way in the oriented graph
$\Gamma_{r,t}^{[k]}$ corresponding the matrix $\Pi_{r,t}^{[k]}(a_0,
a_1, \dots, a_t)$. Then $\overline{\alpha_1} =
\overline{\alpha_{n+1}}$ and there exists $p_j^{(i)}$, $1
\leqslant j \leqslant t-1,$ $1 \leqslant i \leqslant n,$ such that $0
\leqslant p_j^{(i)} \leqslant l_j^{i},$ $\sum\limits_{j=1}^{t-1}
p_j^{(i)} \leqslant k - l_t^{(i)},$ and $l_1^{(i+1)}=l_t^{(i)} +
\sum\limits_{j=1}^{t-1} p_j^{(i)}$, $l_j^{(i+1)} =
l_{j-1}^{(i)}-p_{j-1}^{(i)}$, $2 \leqslant j \leqslant t$, $1
\leqslant i \leqslant n,$
\begin{multline*}
\Bigl( \Pi_{r,t}^{[k]}(a_0,a_1 x, \dots, a_t x^t)
\Bigr)[\mathbf{L}(\overline{\alpha_i} \mid
\mathbf{L}(\alpha_{i+1}))] = \\ =k! \Binom{k}{l_t^{(i)} +
\sum\limits_{j=1}^{t-1} p_j^{(i)}} \cdot \Biggl( \,
\prod_{j=1}^{t-1} \Binom{l_j^{(i)}}{p_j^{(i)}} \Biggr)
a_0^{k-l_t^{(i)} - \sum\limits_{j=1}^{t-1} p_j^{(i)}}
a_1^{p_1^{(i)}} a_2^{p_2^{(i)}} \dots
a_{t-1}^{p_{t-1}^{(i)}}a_t^{l_t^{(i)}}.
\end{multline*}
Hence
\begin{multline*}
\prod_{i=1}^{n}\Bigl( \Pi_{r,t}^{[k]}(a_0,a_1 x, \dots, a_t x^t)
\Bigr)[\mathbf{L}(\overline{\alpha_i} \mid
\mathbf{L}(\alpha_{i+1}))] = \\ =(k!)^n \! \left[ \prod_{i=1}^n
\Binom{k}{l_t^{(i)} + \sum\limits_{j=1}^{t-1} p_j^{(i)}} \! \!
\Binom{l_1^{(i)}}{p_1^{(i)}} \! \! \Binom{l_2^{(i)}}{p_2^{(i)}} \!
\! \dots \Binom{l_{t-1}^{(i)}}{p_{t-1}^{(i)}} \! \cdot
a_0^{k-l_t^{(i)} - \sum\limits_{j=1}^{t-1} p_j^{(i)}}
a_1^{p_1^{(i)}} a_2^{p_2^{(i)}} \dots
a_{t-1}^{p_{t-1}^{(i)}}a_t^{l_t^{(i)}} \right] \times \\ \times
x^{\sum\limits_{i=1}^n (t l_t^{(i)}+ \sum\limits_{j=1}^{t-1}j
p_j^{(i)})} = \biggl( \prod_{i=1}^n (\Pi_{r,t}^{[k]}(a_0,a_1,\dots,
a_t))[\mathbf{L}(\overline{\alpha_i}) \mid
\mathbf{L}(\overline{\alpha_{i+1}})] \biggr) x^{\sum\limits_{i=1}^n
(t l_t^{(i)}+ \sum\limits_{j=1}^{t-1}j p_j^{(i)})}.
\end{multline*}
We will show that $\sum\limits_{i=1}^n (t l_t^{(i)}+
\sum\limits_{j=1}^{t-1}j p_j^{(i)}) = rn.$ From the equalities
$l_j^{(i+1)} = l_{j-1}^{(i)}-p_{j-1}^{(i)},$ $2 \leqslant j
\leqslant t,$ $1 \leqslant i \leqslant n,$ it follows that $j
l_j^{i+1} - l_j^{(i+1)} = (j-1) l_j^{(i+1)} = (j-1) l_{j-1}^{(i)} -
(j-1) p_{j-1}^{(i)}$ and therefore \break $\sum\limits_{i=1}^n
\sum\limits_{j=2}^t j l_j^{(i+1)} - \sum\limits_{i=1}^n
\sum\limits_{j=2}^t l_j^{(i+1)} = \sum\limits_{i=1}^n
\sum\limits_{j=2}^t (j-1) l_{j-1}^{(i)} - \sum\limits_{i=1}^n
\sum\limits_{j=2}^t(j-1) p_{j-1}^{(i)},$ i.d.
$\sum\limits_{i=2}^{n+1} \sum\limits_{j=2}^t j l_j^{(i)} -
\sum\limits_{i=2}^{n+1} \sum\limits_{j=2}^t l_j^{(i)} = \break
=\sum\limits_{i=1}^n \sum\limits_{j=1}^{t-1} j l_j^{(i)} -
\sum\limits_{i=1}^n \sum\limits_{j=1}^{t-1}j p_j^{(i)},$ i.d.
$t\sum\limits_{i=2}^{n+1} l_t^{(i)} + \sum\limits_{j=2}^{t-1} j
l_j^{n+1} + \sum\limits_{i=2}^n \sum\limits_{j=1}^{t-1} jl_j^{(i)} -
\sum\limits_{i=2}^{n+1}(r - l_1^{(i)}) = \sum\limits_{j=1}^{t-1}j
l_j^{(1)} + \sum\limits_{i=2}^n l_1^{(i)} + \break +
\sum\limits_{i=2}^n \sum\limits_{j=2}^{t-1} jl_j^{(i)} -
\sum\limits_{i=1}^n \sum\limits_{j=1}^{t-1} j p_j^{(i)},$ i.d. $t
\Bigl( \sum\limits_{i=1}^n l_t^{(i)} \Bigr) + t (l_t^{n+1}-
l_t^{(1)}) + \Bigl( \sum\limits_{j=2}^{t-1} j l_j^{n+1} \Bigr) - rn
+l_1^{(n+1)} = \sum\limits_{j=1}^{t-1} j l_j^{(1)} - \break -
\sum\limits_{i=1}^n \sum\limits_{j=1}^{t-1} j p_j^{(i)}.$ Hence $t
\sum\limits_{i=1}^n l_t^{(i)} + \sum\limits_{i=1}^n
\sum\limits_{j=1}^{t-1} j p_j^{(i)}= rn + \Bigl(\sum\limits_{j=1}^t
j l_j^{(1)} \Bigr) - \sum\limits_{j=1}^t j l_j^{(n+1)}.$ Since
$\overline{\alpha_1} = \overline{\alpha_{n+1}},$ then $l_j^{(1)} =
l_j^{(n+1)}$ for all $j,$ $1 \leqslant j \leqslant t,$ and therefore
$\sum\limits_{i=1}^n (t l_t^{(i)} + \sum\limits_{j=1}^{t-1} j
p_j^{(i)})=rn.$
\end{proof}
\begin{definition}
Let $n \geqslant 1,$ $P_n(x) = (a_{i,j})_{1 \leqslant i,j \leqslant
n}$,
$$
a_{i,j} = \begin{cases} 1, & \text{if $j-i=1$,} \\
x, & \text{if $j-i \ne 1, \, j-i \equiv 1(\mod n), $} \\
0, & \text{ if $j-i \not\equiv 1 (\mod n)$. }
\end{cases} \eqno(8)
$$
\end{definition}
Then
$$
a_{i,j} = \begin{cases} 1, & \text{if $j-i=1$,} \\
x, & \text{if $i=n, \, j=1 $,} \\
0, & \text{ if $j-i \not\equiv 1 (\mod n)$. }
\end{cases} \eqno(9)
$$
\begin{lemma}
let $s \in \bbZ$, $-n \leqslant s \leqslant n$, $(P_n(x))^s =
(b_{i,j})_{1 \leqslant i,j \leqslant n}$. Then
$$
b_{i,j} = \begin{cases} 1, & \text{if $j-i=s$,} \\
x, & \text{if $j-i \ne s, \, j-i \equiv s(\mod n) $} \\
0, & \text{ if $j-i \not\equiv s (\mod n)$ }.
\end{cases} \eqno(10)
$$
if $0 \leqslant s \leqslant n $ and
$$
b_{i,j} = \begin{cases} 1, & \text{if $j-i=s$,} \\
x^{-1}, & \text{if $j-i \ne s, \, j-i \equiv s(\mod n) $} \\
0, & \text{ if $j-i \not\equiv s (\mod n)$ }
\end{cases} \eqno(11)
$$
if $-n \leqslant s \leqslant -1.$ Hence $(P_n(x))^n = x I_n,$
$$
b_{i,j} = \begin{cases} 1, & \text{if $j-i=s$,} \\
x, & \text{if $j-i = s-n $,} \\
0, & \text{ if $j-i \not\equiv s (\mod n)$ }
\end{cases}  \eqno(12)
$$
if $0 \leqslant s \leqslant n$ and
$$
b_{i,j} = \begin{cases} 1, & \text{if $j-i=s$,} \\
x^{-1}, & \text{if $j-i = s+n $} \\
0, & \text{ if $j-i \not\equiv s (\mod n)$ }
\end{cases}  \eqno(13)
$$
if $-n \leqslant s \leqslant -1.$
\end{lemma}
Let $n \geqslant 1$ and $x \in \bbZ.$ Then, by definition,
$$
\delta_n(x) = \begin{cases}
1, & \text{if $ x \not \equiv 0 (\mod n)$, } \\
0, & \text{ if $ x \equiv 0 (\mod n)$}.
\end{cases}  \eqno(14)
$$
\setcounter{equation}{14}
\begin{theorem}
Let $k \geqslant 1$ and $t \geqslant 0$, let $a_0,a_1, \dots, a_t$
be independent variables. Then for all $n \geqslant 1$
\begin{multline}
\sum_{\sigma \in \SSym(N_{nk})} \sum_{\substack{(\alpha_1, \alpha_2,
\dots, \alpha_{nk}) \in \{ 0,1, \dots, t\}^{nk} \\
\bigl[\frac{\sigma(i)-1}{k} \bigr]- \bigl[\frac{i-1}{k} \bigr]\equiv
\alpha_i (\mod n)}} \biggl( \prod_{i=1}^{nk} a_{\alpha_i} \biggr)
x^{\Card \{ i \in N_{nk} \mid \sigma(i) < i, \alpha_i \not \equiv 0
(\mod n)\} + \sum\limits_{i=1}^{nk} \left[ \frac{\alpha_i}{n}
\right]} = \\ = \sum_{r=0}^{kt}(\Tr((\Pi_{r,t}^{[k]}
(a_0,a_1,\dots,a_t))^n)) x^r
\end{multline}
\end{theorem}
\begin{proof}
Let $\sigma \in \SSym(N_{nk})$, let $i \in N_{nk}$, let $i=ks_i +
p_i$, let $0 \leqslant s_i \leqslant n-1$, let $1 \leqslant p_i
\leqslant k$, let $\bigl[ \frac{\sigma(i) -1}{k}\bigr] =
\bigl[\frac{i-1}{k} \bigr] \equiv \alpha_i(\mod n)$, let $0
\leqslant \alpha_i \leqslant t$. Then there exists $l_i$, $1
\leqslant l_i \leqslant k$, such that $\sigma(i) = k[(s_i +
\alpha_i)_0 \mod n] + l_i,$ where $[(s)_0 \mod n]$ is the least
nonnegative residue of the integer $s \in \bbZ$ modulo $n$. Let
$\sigma(i) > i.$ Then we have
\begin{itemize}
\item[Case 1.] $[(s_i + \alpha_i)_0 \mod n] = s_i, \, l_i >p_i$.
\item[Case 2.] $[(s_i + \alpha_i)_0 \mod n] > s_i.$
\end{itemize}
In the case 1 $\alpha_i \equiv 0(\mod n)$ and $\sigma(i) = k s_i +
l_i.$ In the case $[(s_i \alpha_i)_0 \mod n]= s_i + \alpha_i -
\bigl[ \frac{\alpha_i}{n} \bigr] n$ and $\sigma(i) = k (s_i +
\alpha_i - \bigl[ \frac{\alpha_i}{n} \bigr] n) + l_i.$ Let
$\sigma(i) \leqslant i.$ Then we have
\begin{itemize}
\item[Case 3.] $[(s_i + \alpha_i)_0 \mod n] = s_i$, \, $l_i \leqslant
p_i$.
\item[Case 4.] $[(s_i + \alpha_i)_0 \mod n] < s_i.$
\end{itemize}
In the case 3 $\alpha_i \equiv 0 (\mod n)$ and $\sigma(i) = ks_i +
l_i.$ In the case 4 $[(s_i + \alpha_i)_0 \mod n] = s_i + \alpha_i -
n - \break -\bigl[ \frac{\alpha_i}{n} \bigr] n = s_i + \alpha_i -
(\delta_n(\alpha_i) + \bigl[\frac{\alpha_i}{n} \bigr])n$ and
$\sigma(i) = k(s_i + \alpha_i - (\delta_n(\alpha_i) + \bigl[
\frac{\alpha_i}{n} \bigr])n ) + l_i.$ Let $A_1 = \{ i \in N_{nk}
\mid \alpha_i \equiv 0 (\mod n), l_i > p_i \},$ $A_2 = \{ i \in
N_{nk} \mid [(s_i + \alpha_i)_0 \mod n] > s_i \},$ $A_3 = \{ i \in
N_{nk} \mid \alpha_i \equiv 0 (\mod n), l_i \leqslant p_i \},$ $A_4
= \{ i \in N_{nk} \mid [(s_i + \alpha_i)_0 \mod n] < s_i\}$. Then
\begin{multline*}
\sum_{i=1}^{nk} \Bigl[ \frac{\sigma(i)-1}{k} \Bigr]= \sum_{i \in A_1
\cup A_3} s_i + \sum_{i \in A_2} s_i + \alpha_i - \Bigl[
\frac{\alpha_i}{n} \Bigr] n + \sum_{i \in A_4} s_i + \alpha_i -
\Bigl(\delta_n(\alpha_i) + \Bigl[ \frac{\alpha_i}{n} \Bigr] \Bigr)n
=\\= \sum_{i \in A_1 \cup A_3} s_i + \alpha_i - \alpha_i + \sum_{i
\in A_2} s_i + \alpha_i - \Bigl[ \frac{\alpha_i}{n} \Bigr] n +
\sum_{i \in A_4} s_i + \alpha_i - (\delta_n (\alpha_i) +
\Bigl[\frac{\alpha_i}{n} \Bigr])n = \\ =\sum_{i=1}^{nk} s_i +
\sum_{i=1}^{nk} \alpha_i - \sum_{i \in A_1 \cup A_3} \Bigl[
\frac{\alpha_i}{n} \Bigr] n - \sum_{i \in A_2}
\Bigl[\frac{\alpha_i}{n} \Bigl] n - \sum_{i \in A_4} \biggl[
\frac{\alpha_i}{n} \biggr] n - \sum_{i \in A_4} \delta_n(\alpha_i)n
= \\ = \sum_{i=1}^{nk} s_i + \sum_{i=1}^{nk} \alpha_i -
\sum_{i=1}^{nk} \biggl[ \frac{\alpha_i}{n} \biggr] n -
\sum_{\sigma(i) < i, \alpha_i \not\equiv 0 (\mod n)}
\delta_n(\alpha_i) n = \\=  \sum_{i=1}^{nk} s_i + \sum_{i=1}^{nk}
\alpha_i - \biggl( \biggl( \sum_{i=1}^{nk} \biggl[
\frac{\alpha_i}{n} \biggr] \biggr) + \Card \{ i \in N_{nk} \mid
\sigma(i) < i, \alpha_i \not\equiv 0 (\mod n) \} \biggr) n,
\end{multline*}
$\sum\limits_{i=1}^{nk} \bigl[\frac{i-1}{k} \bigr] =
\sum\limits_{i=1}^{nk} s_i.$ But $\sum\limits_{i=1}^{nk} \bigl[
\frac{\sigma(i)-1}{k} \bigr]= \sum\limits_{i=1}^{nk} \bigl[
\frac{i-1}{k} \bigr],$ and therefore
$$
\sum_{i=1}^{nk} \alpha_i= (\Card \{ i \in N_{nk} \mid \sigma(i) < i,
\alpha_i \not\equiv 0 (\mod n)\} +
\sum_{i=1}^{nk}\bigl[\frac{\alpha_i}{n} \bigr] )n.
$$
Since
$$
\sum_{\sigma \in
\SSym(N_{nk})}\sum_{\substack{{(\alpha_1,\alpha_2,\dots,\alpha_{nk})}
\in \{ 0,1,\dots,t \}^{nk} \\ \left[ \frac{\sigma(i)-1}{k} \right] -
\left[ \frac{i-1}{k} \right] \equiv \alpha_i (\mod n) }} \biggl(
\prod_{i=1}^{nk} a_{\alpha_i} \biggr) = \per \biggl( \biggl(
\sum_{i=0}^t a_i P_n^i \biggr) \otimes J_k \biggr)
$$
for all $n \geqslant 1$, then from the theorem 3 it follows that
$$
\sum_{\sigma \in \SSym(N_{nk})}
\sum_{\substack{{(\alpha_1,\alpha_2,\dots,\alpha_{nk})} \in \{
0,1,\dots,t \}^{nk} \\ \left[ \frac{\sigma(i)-1}{k} \right] - \left[
\frac{i-1}{k} \right] \equiv \alpha_i (\mod n) }} \biggl(
\prod_{i=1}^{nk} a_{\alpha_i} x^{\alpha_i} \biggr) = \sum_{r=0}^{kt}
(\Tr(\Pi_{r,t}^{[k]}(a_0,a_1,\dots,a_t))^n)x^{rn}, \; n \geqslant 1.
$$
Hence and from the equality $\sum\limits_{i=1}^{nk} \alpha_i =
\bigl(\Card \{ i \in N_{nk} \mid \sigma(i)< i, \alpha_i \not\equiv 0
(\mod n) \} + \sum\limits_{i=1}^{nk} \bigl[\frac{\alpha_i}{n} \bigr]
\bigr)n$ it follows that for all $n \geqslant 1$
\begin{multline*}
\sum_{\sigma \in \SSym(N_{nk})}
\sum_{\substack{{(\alpha_1,\alpha_2,\dots,\alpha_{nk})} \in \{
0,1,\dots,t \}^{nk} \\ \left[ \frac{\sigma(i)-1}{k} \right] - \left[
\frac{i-1}{k} \right] \equiv \alpha_i (\mod n) }} \biggl(
\prod_{i=1}^{nk} a_{\alpha_i} \biggr) \cdot x^{\Card \{ i \in N_{nk}
\mid \sigma(i)< i, \alpha_i \not\equiv 0 (\mod n) \} +
\sum\limits_{i=1}^{nk} \left[ \frac{\alpha_i}{n} \right]} = \\ =
\sum_{r=0}^{kt}(\Tr( \Pi_{r,t}^{[k]}(a_0,a_1,\dots,a_t))^n ) x^r.
\end{multline*}
\end{proof}
\begin{theorem}
Let $k \geqslant 1$ and $t \geqslant 0$, let $a_0,a_1,\dots,a_t$ be
elements in a commutative ring with unity. Then for all $n \geqslant
t+1$
\begin{multline}
\sum_{\sigma \in \SSym(N_{nk})} \sum_{r=0}^{kt} x^r
\sum_{\substack{{(\alpha_1,\alpha_2,\dots,\alpha_{nk})} \in  \{
0,1,\dots,t \}^{nk} \\ \left[ \frac{\sigma(i)-1}{k} \right] - \left[
\frac{i-1}{k} \right] \equiv \alpha_i (\mod n) }} \biggl(
\prod_{i=1}^{nk} a_{\alpha_i} \biggr) \cdot x^{\Card \{ i \in N_{nk}
\mid \sigma(i)< i, \alpha_i \ne 0 \}} = \\ = \sum_{r=0}^{kt}(\Tr((
\Pi_{r,t}^{[k]}(a_0,a_1,\dots,a_t))^n )) x^r.
\end{multline}
\end{theorem}
\begin{proof}
Since $0 \leqslant \bigl[\frac{\alpha_i}{n} \bigr] \leqslant \bigl[
\frac tn \bigr] =0$, then $\bigl[\frac{\alpha_i}{n} \bigr] =0$ for
all $i \in N_{nk}$. If $t<n$ and $\alpha_i \equiv 0 (\mod n)$, then
$\alpha_i = 0$.
\end{proof}
\begin{theorem}
Let $k \geqslant 1$ and $t \geqslant 0$, let $a_0,a_1, \dots, a_t$
be elements in a commutative ring with unity. Then for all $n
\geqslant 1$
\begin{equation}
\per \Bigl( \Bigl( \sum_{i=0}^t a_i (P_n(x))^i \Bigr) \otimes J_k
\Bigr) = \sum_{r=0}^{kt} ( \Tr (( \Pi_{r,t}^{[k]} (a_0, a_1, \dots,
a_t))^n ) x^r.
\end{equation}
\end{theorem}
\begin{lemma}
Let $k,t \geqslant 1$. Then $|G_{r,t}^{[k]}| = \binom{t+r-1}{r}$ if
$0 \leqslant r \leqslant k,$ and $G_{r,t}^{[k]} = \sum\limits_{i=0}^k
|G_{r-i, t-1}^{[k]} |$ if $k \leqslant r \leqslant kt.$
$|G_{r,t}^{[k]}| = |G_{kt-r,t}^{[k]}|$, $|G_{r,t}^{[k]}| =
\sum\limits_{i=0}^{\min(k,r)} |G_{r-i, t-1}^{[k]}|$, $0 \leqslant r
\leqslant kt$.
\end{lemma}
\begin{theorem}
Let $k \geqslant 1$ and $t \geqslant 0$, let $0 \leqslant r
\leqslant kt$, let $a_0, a_1, \dots, a_t$ be elements in a
commutative ring with unity. Then
\begin{equation}
\det(x I_{|G_{r,t}^{[k]}|} - \Pi_{r,t}^{[k]}(a_0,a_1, \dots, a_t)) =
\det( x I_{|G_{kt-r,t}^{[k]}|} - \Pi_{kt-r,t}^{[k]}(a_t,a_{t-1},
\dots, a_0)).
\end{equation}
\end{theorem}
\begin{theorem} Let $k\geqslant 1$, let $a_0$, $a_1$,~\dots,~$a_t$ be
independent variables
over a field $F$ of characteristic~$0$. Then the polynomials
$$\det(I_{|G_{\smash{r,t}}^{[k]}|}-x\Pi_{\smash{r,t}}^{[k]}(a_0,a_1,\dots,a_t)),
\; 0\leqslant r\leqslant kt \mbox{ and }
\det(I_{(k+1)^t}-x
K_t^{\smash{[k]}}(a_0,a_1,\dots,a_t))
$$
are irreducible over the field of rational functions
$F(a_0,a_1,\dots,a_t)$.
\end{theorem}

\begin{theorem}
On the set $G_{\smash{r,t}}^{[k]}$, $1\leqslant r\leqslant kt$, we
define the injective mapping~$\phi_{\smash{r,t}}$ as follows way: if
$\overline\alpha=( 1^{\langle l_1\rangle}, 2^{\langle
l_2\rangle},\dots, t^{\langle l_t\rangle})\in
G_{\smash{r,t}}^{[k]}$, then $\phi_{\smash{r,t}}(\overline\alpha)=(
1^{\langle l_t\rangle}, 2^{\langle l_1\rangle}, 3^{\langle
l_2\rangle},\dots,t^{\langle l_{t-1}\rangle})$. Then
$\phi_{\smash{r,t}}$ is the substitution of the set
$G_{\smash{r,t}}^{[k]}$ and
\begin{gather}
\det(\Pi_{\smash{r,t}}^{[k]}(a_0,a_1,\dots,a_t))=(\sign
\phi_{\smash{r,t}})\prod\limits_{i=0}^{\min(k,r)}\left( k!\binom
kia_0^{k-i}a_t^i\right)^{|G_{\smash{r-i,t-1}}^{[k]}|},\\
\det(K_{\smash{t}}^{[k]}(a_0,a_1,\dots,a_t)) =
\left(\prod\limits_{r=1}^{kt}\sign\phi_{\smash{r,t}}\right)
\prod\limits_{i=0}^{k}\left( k!\binom
kia_0^{k-i}a_t^i\right)^{(k+1)^{t-1}},\\
\per(\Pi_{\smash{r,t}}^{[k]}(a_0,a_1,\dots,a_t))=
\prod\limits_{i=0}^{\min(k,r)}\left(
k!\binom kia_0^{k-i}a_t^i\right)^{|G_{\smash{r-i,t-1}}^{[k]}|}.
\end{gather}
\end{theorem}

\begin{theorem}
Let $0 \leqslant r \leqslant t$, let $a_{-r}$, $a_{-r+1}$, \dots,
$a_{-r+t}$ be elements in a commutative ring with unity. Then
\begin{multline}
 1+ \sum_{n=1}^{\infty}\ls R \ls x;\ls \sum\limits_{i=0}^{t} a_{-r+i}
T_n^{(-r+i)} \rs \otimes J_k \rs \rs y^n = (-1)^{ \mathbf{ L }
(1^{<k>}, \dots, r^{<k>})} \times \\ \times \biggl[ \! \! \!\!\!
\sum_{\hbox to
40pt{$ \substack{\ogam \in \bigcup_{l=rk}^{kt} G_{l,t}^{[k]} \\
\{\ogam\} \supset \{ 1^{\langle k \rangle}, 2^{\langle k \rangle},
\ldots, r^{\langle k \rangle} \} }$}} \! \!
\!(-1)^{\mathbf{L}(\overline \gamma) }\det \bigl((I_{(k+1)^t}-y
K_t^{[k]}( a_{-r}x, a_{-r+1}x, \ldots, a_{-r+t}x)) \times \\ \times
( \mathbf{L} (\ogam)|  \mathbf{ L } (1^{<k>}, 2^{<k>}, \ldots,
r^{<k>})) \bigr)\biggr] \cdot (\det(I_{(k+1)^t}-y K_t^{[k]}(
a_{-r}x,a_{-r+1}x, \ldots, a_{-r+t}x)))^{-1},
\end{multline}
\begin{multline}
\sum_{n=1}^{\infty} \ls R \ls x;\ls \sum_{i=0}^{t} a_{-r+i}
P_n^{-r+i} \rs \otimes J_k \rs \rs y^n =
\\ =-y \biggl( \frac{\dif}{\dif y} \det(I_{(k+1)^t}-y K_t^{[k]}(
a_{-r}x, a_{-r+1}x, \ldots, a_{-r+t}x)) \biggr) \times
\\ \times \bigl(\det(I_{(k+1)^t}-y K_t^{[k]}( a_{-r}x, a_{-r+1}x, \ldots,
a_{-r+t}x)) \bigr)^{-1},
\end{multline}
\begin{multline}
 1+ \sum_{n=1}^{\infty} \ls
\per \ls \ls
\sum_{i=0}^{t} a_{-r+i} T_n^{(-r+i)} \rs \otimes J_k \rs \rs x^n = \\
 =\biggl[ \det \Bigl( \Bigl( I_{| G_{rk,t}^{[k]}|} - x\Pi_{rk,t}^{[k]}( a_{-r},
a_{-r+1}, \ldots, a_{-r+t} ) \Bigr) \bigl( \mathbf{ L } (1^{<k>}, \ldots, r^{<k>})|
\mathbf{ L } (1^{<k>}, \ldots, r^{<k>}) \bigr) \Bigr) \biggr] \times \\
 \times \biggl(\det \Bigl( I_{| G_{rk,t}^{[k]}|} - x
\Pi_{rk,t}^{[k]}(a_{-r}, a_{-r+1}, \ldots, a_{-r+t}) \Bigl)
\biggr)^{-1},
\end{multline}
\begin{multline}
 \shoveleft{ \suml{n=1}{\infty} \left( \per \ls \left(
\sum_{i=0}^{t} a_{-r+i} P_n^{-r+i} \right) \otimes J_k \rs \right) x^n = }\\
= -x  \sum_{l=0}^{kt} \biggl[ \frac{\dif}{\dif x} \Bigl( \det \Bigl(
I_{| G_{l,t}^{[k]}|} -x\Pi_{l,t}^{[k]} (a_{-r},
a_{-r+1}, \ldots, a_{-r+t}) \Bigr) \Bigr) \biggr] \times \\
\times \Bigl(\det \Bigl( I_{| G_{l,t}^{[k]}|} - x
\Pi_{l,t}^{[k]}(a_{-r}, a_{-r+1}, \ldots, a_{-r+t}) \Bigl)
\Bigr)^{-1}.
\end{multline}
\end{theorem}
\begin{theorem}
Let $k \geqslant  2,$ let $ 1 \leqslant t_1 \leqslant t_2 \leqslant
\ldots \leqslant t_k = t,$ let greatest common divisor of $t_1, t_2,
\ldots, t_k$ is equal to 1. Let $a_{t_1}, a_{t_2}, \ldots, a_{t_k}$
be independent variables over a field $F$ of characteristic $0$,
$f(x) = x^t + \suml{i=1}{k} a_{t_i}x^{t-t_i}.$ Then the Galois group
of the polynomial $f(x)$ over the rational function field
$F(a_{t_1}, a_{t_2}, \ldots, a_{t_k})$ is isomorphic to the
symmetric group $\SSym(t)$ of degree $t$.
\end{theorem}
\begin{theorem}
Let $k \leqslant 2$, let $1 \leqslant t_1 \leqslant t_2 \leqslant
\ldots \leqslant t_k = t,$ let greatest common divisor of $t_1, t_2,
\ldots, t_k$ equals 1. Assume that $t$ is even, $t_i$ is odd for all
$i, \, 1 \leqslant i \leqslant k-1$, Let $a_{t_1}, a_{t_2}, \ldots,
a_{t_k}$ be independent variables over a field $F$ of characteristic
$2$, $f(x) = x^t + \suml{i=1}{k} a_{t_i}x^{t-t_i}.$ Then the Galois
group of the polynomial $f(x)$ over the rational function field
$F(a_{t_1}, a_{t_2}, \ldots, a_{t_k})$ is isomorphic to the
symmetric group $\SSym(t)$ of degree $t$.
\end{theorem}

For convenience of the reader we shall remind [3, p.80] definition
of the matrix
$$
D_r(a_0, a_1, \ldots, a_t) = (a_{\oalp, \obet})_{\oalp, \obet \in
\Q_{r,t}}, \, 0 \leqslant r \leqslant t. \quad D_0(a_0)=a_0.
$$
Let $1 \leqslant r \leqslant t, \:  \oalp = (i_1, i_2, \ldots i_r)
\in \Q_{r,t}.$ If $i_r < t$ then
\begin{equation}
a_{\oalp, \obet} =
\begin{cases}
a_0, &\text {if $\obet = (i_1+1, i_2+1, \ldots, i_r+1)$,} \\
(-1)^k a_{i_k}, &\text{if $\obet = (1,i_1+1, \ldots, \widehat{i_k+1},
\ldots, i_r+1 ), \, 1 \leqslant k \leqslant r$,} \\
0, &\text{otherwise.}
\end{cases}
\end{equation}
If $i_r=t,$ then
\begin{equation}
a_{\oalp, \obet} = \begin{cases}
 (-1)^r a_t, &\text{ if $\obet = (1,i_1+1,i_2+1, \ldots, i_{r-1}+1 ),$ } \\
0, &\hbox{if $\obet \ne (1,i_1+1,i_2+1, \ldots, i_{r-1}+1 ).$}
\end{cases}
\end{equation}
It follows that
\begin{equation}
D_1(a_0,a_1, \ldots, a_t) = \Pi_{1,t}^{[1]}(a_0, -a_1, \ldots, -a_t
).
\end{equation}

\begin{theorem}
Let $k \! \geqslant \! 2,$ let $1 \! \leqslant \! t_1 \! \leqslant
\! t_2 \! \leqslant \! \ldots \! \leqslant \!t_k = t,$ let greatest
common divisor of $t_1, t_2, \ldots, t_k$ is equal to 1. Let
$a_0,a_{t_1}, a_{t_2}, \ldots, a_{t_k}$ be independent variables
over a field $F$ of characteristic $0$, let $a_i=0$ for all $i \in
N_{t-1} \setminus \{ t_1,t_2, \ldots, t_{k-1}\}$. Then the
polynomial $\det(xI_{\binom{t}{r}} - D_r(a_0, a_1, \ldots, a_t))$ is
irreducible over the rational function field $F(a_0,a_{t_1},
a_{t_2}, \ldots, a_{t_k}).$
\end{theorem}

\begin{theorem}
Let $t \! \geqslant \!  2,$ let $1 \! \leqslant \! t_1 \! \leqslant
\! t_2 \! \leqslant \! \ldots \! \leqslant \! t_k = \!t,$ let
greatest common divisor of \, $t_1, t_2, \ldots, t_k$ is equal to 1.
Assume that $t$ is even, $t_i$ is odd for all $i, \, 1 \leqslant i
\leqslant k-1.$ Let $0 \leqslant r \leqslant t,$ let $a_0, a_{t_1},
a_{t_2}, \ldots, a_{t_k}$ be independent variables over a field $F$
of characteristic $0$, let $a_i=0$ for all $i \in N_{t-1} \setminus
\{ t_1,t_2, \ldots, t_{k-1}\}.$ Then the polynomial
$\det(xI_{\binom{t}{r}} - \Pi_{r,t}^{[1]}(a_0, a_1, \ldots, a_t))$
is irreducible over the rational function field $F(a_0,a_{t_1},
a_{t_2}, \ldots, a_{t_k}).$
\end{theorem}

\begin{theorem}
Let $t \geqslant 1$, let $0 \leqslant s \leqslant t-1$ and $1
\leqslant n \leqslant t$, let $a_0,a_1, \dots, a_t$ be independent
variables. Then
\begin{multline}
\Tr \Bigl( \Bigl( \frac{\partial}{\partial a_0} D_s(a_0, a_1, \dots,
a_t ) \Bigr) \cdot (D_s (a_0, a_1, \dots, a_t))^{n-1} \Bigr) = \\ =-
\Tr \Bigl( \Bigl( \frac{\partial}{\partial a_n} D_{s+1}(a_0, a_1,
\dots, a_t) \Bigr) \cdot (D_{s+1} (a_0, a_1, \dots, a_t))^{n-1}
\Bigr)
\end{multline}
\end{theorem}
\begin{theorem}
Let $t \geqslant 1$ and $1 \leqslant n \leqslant t$, let $a_0, a_1,
\dots, a_t$ be independent variables. Then
\begin{multline}
\sum_{s=0}^t (-1)^s \Tr \Bigl( \Bigl( \frac{\partial}{\partial a_0}
D_s(a_0, a_1, \dots, a_t ) \Bigr) \cdot (D_s (a_0, a_1, \dots,
a_t))^{n-1} \Bigr) = \\ = \sum_{s=0}^t(-1)^s \Tr \Bigl( \Bigl(
\frac{\partial}{\partial a_n} D_s(a_0, a_1, \dots, a_t) \Bigr) \cdot
(D_s (a_0, a_1, \dots, a_t))^{n-1} \Bigr).
\end{multline}
\end{theorem}
\begin{proof}
Since $D_0(a_0, a_1, \dots, a_t) = a_0$, $D_t(a_0, a_1, \dots, a_t)
= (-1)^t a_t$, then $\frac{\partial}{\partial a_0 } D_t (a_0, a_1,
\dots, a_t)=0$, $\frac{\partial}{\partial a_n} D_0(a_0, a_1, \dots,
a_t) = 0$ and from the theorem 27 it follows, that
\begin{multline*}
\sum_{s=0}^t (-1)^s \Tr \Bigl( \Bigl( \frac{\partial}{\partial a_0}
D_s(a_0, a_1, \dots, a_t ) \Bigr) \cdot (D_s (a_0, a_1, \dots,
a_t))^{n-1} \Bigr) = \\
=\sum_{s=0}^{t-1} (-1)^s \Tr \Bigl( \Bigl( \frac{\partial}{\partial
a_0} D_s(a_0, a_1, \dots, a_t ) \Bigr) \cdot (D_s (a_0, a_1, \dots,
a_t))^{n-1} \Bigr) = \\
=\sum_{s=0}^{t-1} (-1)^{s+1} \Tr \Bigl( \Bigl(
\frac{\partial}{\partial a_n} D_{s+1}(a_0, a_1, \dots, a_t ) \Bigr)
\cdot (D_{s+1} (a_0, a_1, \dots, a_t))^{n-1} \Bigr) = \\
=\sum_{s=1}^t (-1)^s \Tr \Bigl( \Bigl( \frac{\partial}{\partial a_n}
D_s(a_0, a_1, \dots, a_t ) \Bigr) \cdot (D_s (a_0, a_1, \dots,
a_t))^{n-1} \Bigr) = \\
=\sum_{s=0}^t (-1)^s \Tr \Bigl( \Bigl( \frac{\partial}{\partial a_n}
D_s(a_0, a_1, \dots, a_t ) \Bigr) \cdot (D_s (a_0, a_1, \dots,
a_t))^{n-1} \Bigr)
\end{multline*}
\end{proof}

\begin{theorem}
Let $t \geqslant 2$, let $a_0, a_1, \dots, a_t$ be elements in a
commutative ring, let $1 \leqslant r \leqslant t-1$, let
$\overline{\alpha}, \overline{\beta} \in Q_{r,t}$, let $1 \leqslant
n \leqslant t-1$. Assume that $\max \{ \overline{\beta} \} \geqslant
n+1$, $\overline{\gamma} \in Q_{l,t}$, $ \{\overline{\gamma} \}
\subseteq \{ \overline{\beta} \}$, $\min\{ \overline{\gamma} \}
\geqslant n+1$ and $\{ \overline{\alpha} \} \not\supseteq \{
\overline{\gamma} -n \}$. Then
\begin{equation}
\bigl( D_r (a_0, a_1,  \dots a_t ) \bigr)^n [
\mathbf{L}_{r,t}(\overline{\alpha}) \mid
\mathbf{L}_{r,t}(\overline{\beta})] = 0.
\end{equation}
\end{theorem}

\begin{theorem}
Let $t \geqslant 2$, let $a_0, a_1, \dots, a_t$ be elements in a
commutative ring, let $1 \leqslant r \leqslant t$,
$\overline{\alpha}, \overline{\beta} \in Q_{r,t}$, let $1 \leqslant
n \leqslant t-1$. Assume, that $\max \{ \overline{\beta} \}
\geqslant n+1$, $\overline{\gamma} \in Q_{l,t}$, $
\{\overline{\gamma} \} \subseteq \{ \overline{\beta} \}$, $\min\{
\overline{\gamma} \} \geqslant n+1$ and $\{ \overline{\alpha} \}
\supseteq \{ \overline{\gamma} -n \}$. Then
\begin{multline}
( D_r (a_0, a_1,\dots,a_t ))^n [ \mathbf{L}_{r,t}(\overline{\alpha})
\mid \mathbf{L}_{r,t}(\overline{\beta})] =(-1)^{\Card \{ (i,j) \in
\{ \overline{\gamma}-n \} \times ( \{ \overline{\alpha} \} \setminus
( \{ \overline{\beta} - n \} \cap \{\overline{\alpha}\}) )\mid i < j
\} } \times \\ \times (D_{r-l}(a_0, a_1, \dots, a_t))^n
[\mathbf{L}_{r-l,t}(\overline{\alpha} \setminus (\overline{\gamma}
-n)) \mid
\mathbf{L}_{r-l,t}(\overline{\beta}\setminus\overline{\gamma})]
\end{multline}
\end{theorem}
\begin{theorem}
Let $K$ be a field of characteristic 0, let $x_1,x_2, \dots, x_m$ be
algebraically independent variables over $K$ and $E$ be a algebraic
extension of the field $K(x_1, x_2, \dots, x_m)$. Let $D_{x_1},
D_{x_2}, \dots, D_{x_m}$ be the extensions, respectively, of the
differentiations $\frac{\partial}{\partial x_1},
\frac{\partial}{\partial x_2}, \dots, \frac{\partial}{\partial x_m}$
in the field of rational functions $K(x_1, x_2, \dots, x_m)$ up to
the differentiations in the field $E$. Assume that for some $y \in
E$ will be $D_{x_1}(y)=D_{x_2}(y)=\dots=D_{x_m}(y)=0$. Then the
element $y$ is algebraic over the field $K$.
\end{theorem}
\begin{proof}
Let $f(x)$ be the minimal polynomial of the element $y$ over the
field $K(x_1, x_2, \dots, x_m)$, $f(x)=x^k + \sum\limits_{i=1}^k a_i
x^{k-i}$, $a_i \in K(x_1, x_2, \dots, x_m)$, $1 \leqslant i
\leqslant k$. Then $y^k + \sum\limits_{i=1}^k a_i y^{k-i}=0$ and
applying differentiation $D_{x_j}$, $1 \leqslant j \leqslant m$, we
obtain that $k y^{k-1} D_{x_j}(y) + \sum\limits_{i=1}^k a_i (k-i)
y^{k-i-1} D_{x_j}(y) + \break + \sum\limits_{i=1}^k(D_{x^j}(a_i))
y^{k-i}=0$. Since by condition $D_{x_j} (y) = 0$, $1 \leqslant j
\leqslant m$, then $\sum\limits_{i=1}^k (D_{x_j}(a_i)) y^{k-i}=0$.
But $D_{x_j}(a_j) = \frac{\partial}{\partial x_j} a_i \in K(x_1,
x_2, \dots, x_m)$ and therefore the element $y$ vanish by the
polynomial $\sum\limits_{i=1}^k \bigl( \frac{\partial}{\partial x_j}
a_i \bigr) x^{k-i}$ over the field $K(x_1, x_2, \dots, x_m)$ and
$\deg \bigl( \sum\limits_{i=1}^k \bigl( \frac{\partial}{\partial
x_j} a_i \bigr) x^{k-i} \bigr) \leqslant k-1$. If
$\frac{\partial}{\partial x_j} a_i \ne 0$ for some $i$, $ 1
\leqslant i \leqslant k$, then $\sum\limits_{i=1}^k \bigl(
\frac{\partial}{\partial x_j} a_i \bigr) x^{k-i}$ is nonzero
polynomial of degree $ \leqslant k-1$, that contradict assumption,
that $k$ is degree of the minimal polynomial of the element $y$ over
the field $K(x_1, x_2, \dots, x_m)$. Therefore
$\frac{\partial}{\partial x_j} a_i =0$ for all $i,j$, $1 \leqslant i
\leqslant k$, $1 \leqslant j \leqslant m$. Hence $a_i \in K$ for all
$i$, $1 \leqslant i \leqslant k$.
\end{proof}
\begin{lemma}
Let $K$~ be a field of characteristic 0, $f(x,y) \in K[[x,y]],$
$f(x,y) = \sum\limits_{i,j\geqslant 0,\ i,j\in\bbZ} a_{i,j} x^i
y^j$. Assume that $\frac{\partial f(x,y)}{\partial x} =
\frac{\partial f(x,y)}{\partial y}$. Then
\begin{gather}
a_{i,j} = \binom{i+j}i a_{0, i+j},\\
a_{i,j} = a_{j,i},\\
f(x,y) = \sum\limits_{k=0}^\infty a_{0,k}(x+y)^k
\end{gather}
\end{lemma}
\begin{proof}
Since $\frac{\partial f(x,y)}{\partial x} = \sum\limits_{\substack{i\geqslant 1,
j\geqslant 0 \\ i,j\in \bbZ}} i a_{i,j}
x^{i-1} y^j$, $\frac{\partial f(x,y)}{\partial y} =
\sum\limits_{\substack{i\geqslant 0, j\geqslant 1\\ i,j\in \bbZ}} j a_{i,j} x^{i} y^{j-1}$,
 then provided that
$\frac{\partial f(x,y)}{\partial x} =\break = \frac{\partial
f(x,y)}{\partial y}$ it follows that $i a_{i,j} = (j+1)a_{i-1,
j+1},$ $ i \geqslant 1$. Hence $a_{i,j} = \frac{j+1}i a_{i-1, j+1} =
\frac{j+1}i \cdot \frac{j+2}{i-1} a_{i-2,j+2} =\break= \frac{j+1}i
\cdot \frac{j+2}{i-1} \cdot \frac{j+3}{i-2} a_{i-3, j+3} =
\frac{j+1}i \cdot \frac{j+2}{i-1} \cdot \frac{j+3}{i-2} \ldots
\frac{j+i}{i-(i-1)}a_{i-i, j+i} = \break =\frac{(j+i)(j+i-1) \ldots
(j+1)}{i!} a_{0,j+i} = \binom{j+i}{i}a_{0,j+i}$. By the equality
(33) it follows that $a_{j,i} = \binom{j+i}{j} a_{0,j+i} = \break
=\binom{i+j}{i} a_{0,i+j} = a_{i,j},$ $f(x,y) =
\sum\limits_{\substack{i,j \in \bbZ\\
\, i,j \geqslant 0 }} a_{0,i+j} \binom{i+j}{i} x^i y^j =
\sum\limits_{k=0}^\infty a_{0,k} \sum\limits_{\substack{i+j=k \\ i
\geqslant 0, \, j \geqslant 0, \, i,j \in \bbZ }} \binom{i+j}{i} x^i
y^j = \break = \sum\limits_{k=0}^\infty a_{0,k} \sum\limits_{i=0}^k
 \binom{k}{i} x^i y^{k-i} = \sum\limits_{k=0}^\infty a_{0,k} (x+y)^k.$

Lemma 3 is proved.
\end{proof}

\begin{lemma}
Let $K$~be a field of characteristic  0, $f(x_1,x_2, \dots, x_m) \in
K[[x_1,x_2, \dots, x_m]],$
$$
f(x_1, x_2, \dots, x_m) = \sum\limits_{i_1,i_2,\ldots,i_m\geqslant
0; i_1,i_2, \ldots, i_m \in \bbZ} a_{i_1,i_2, \dots, i_m}
x_1^{i_1},x_2^{i_2}, \dots, x_m^{i_m}, \; m \geqslant 2.
$$
Assume that $\frac{\partial f}{\partial x_1} = \frac{\partial
f}{\partial x_2} = \ldots = \frac{\partial f}{\partial x_m}$. Then
\begin{gather}
a_{i_1,i_2, \dots, i_m} = \frac{(i_1+i_2 +\dots + i_m)!}{i_1! i_2! \dots i_m!} a_{0,0, \ldots, 0, i_1+i_2+ \dots + i_m}, \\ 
a_{i_1,i_2, \dots, i_m} = a_{i_{\sigma(1)},i_{\sigma(2)}, \dots, i_{\sigma(m)}}, \; \sigma \in \SSym(N_m) \\
f(x_1,x_2, \dots, x_m) = \sum\limits_{k=0}^\infty a_{0,0, \ldots,
0,k}(x_1 +x_2 + \dots + x_m)^k
\end{gather}
\end{lemma}
\begin{proof}
Case $m=2$ is proved in the lemma 3. We use induction on $m
\geqslant 2$. Since
$$
 K[[x_1, x_2, \dots, x_m]] =(K[[x_m]])[[x_1, x_2,
\dots, x_{m-1}]],
$$
then
$$
f(x_1, x_2, \dots, x_m) = \sum\limits_{\substack
{j_1,j_2,\ldots,j_{m-1} \geqslant 0; \\ j_1, j_2, \ldots, j_{m-1}
\in \bbZ}} b_{j_1, j_2, \ldots, j_{m-1}} x_1^{j_1} x_2^{j_2} \ldots
x_{m-1}^{j_{m-1}},
$$
where $b_{j_1,j_2, \dots, j_{m-1}} \in K[x_m].$ Let $m \geqslant 3$.
By the induction hypothesis for $m-1$ it follows that
\begin{equation}
f(x_1, x_2, \dots, x_m) = \sum_{k=0}^{\infty} b_{0,0, \dots,0,k}
(x_1 + x_2 + \dots + x_{m-1})^k
\end{equation}
Let $y = x_1 + x_2 + \dots + x_{m-1}.$ Then $y$ and $x_m$~are
algebraically independent elements over the field $K$ and from
equality (39) it follows that $f(x_1, x_2, \dots, x_m) =
\sum\limits_{i,j \geqslant 0; \,  i,j \in \bbZ } c_{i,j} y^i x_m^j,$
$c_{i,j} \in K.$ Since $\frac{\partial y}{\partial x_1} = 1,$
$\frac{\partial f}{\partial x_1} = \frac{\partial f}{\partial y}
\cdot \frac{\partial y}{\partial x_1} = \frac{\partial f}{\partial
y},$ $\frac{\partial f}{\partial x_1} = \frac{\partial f}{\partial
x_m}$, then$\frac{\partial f}{\partial y} = \frac{\partial
f}{\partial x_m}$. By the lemma 3 it follows that $f(x_1, x_2,
\ldots, x_m) = \sum\limits_{k=0}^{\infty} c_{0,k} (y+x_m)^k =
\sum\limits_{k=0}^{\infty} c_{0,k} (x_1+x_2+ \ldots +x_m)^k$ and
therefore $c_{0,k} = a_{0,0,\ldots,0,k}$. The equality (36) follows
directly from(38). The equality (37) follows directly from (36).

Lemma 4 is proved.
\end{proof}

\begin{lemma}
    Let $A$~be an $m \times n$ rectangular matrix over a commutative
    ring and let
     $m \leqslant n$. Let $(H_1, H_2,
    \ldots, H_s))$~be a fixed ordered partition of the set
    $N_m$. Then
    \begin{gather}
\per(A) = \sum_{\substack{(K_1,K_2,\ldots,K_s), \, K_i \subseteq N_n,\\
|K_i|=|H_i|, \, 1 \leqslant i \leqslant s \\ K_i \cap K_j =
\varnothing, \, 1 \leqslant i < j \leqslant s}}
\prod_{i=1}^s \per(A[H_i \mid K_i]), \label{40}\\
R(x;A)= \sum_{\substack{(K_1,K_2,\ldots,K_s), \, K_i \subseteq N_n,\\
|K_i| \leqslant |H_i|, \, 1 \leqslant i \leqslant s \\ K_i \cap K_j
= \varnothing, \, 1 \leqslant i < j \leqslant s}} \biggl(
\prod_{i=1}^s \per(A[H_i \mid K_i])
\biggr)x^{\sum\limits_{i=1}^s|K_i|} \label{41}
\end{gather}
\label{lem:5}
\end{lemma}
    Proof of the theorem~4, independent of general theory.

    Let $A = (a_0 I_n + a_1 P_n) \otimes J_k$, $H_i = (i-1)k + N_k$,
    $i \leqslant i \leqslant n$. Then from lemma 5 it follows that
    \begin{multline}
        \per(A) = \sum_{\substack{(K_1,K_2, \ldots,K_n), \, K_i
        \subset (i-1)k + N_{2k}, \, 1 \leqslant i \leqslant n-1, \\
        K_n \subset ( (n-1)k + N_k) \cup N_k, \, K_i\cap K_j =
        \varnothing, \, 1 \leqslant i < j \leqslant n, \, |K_i| = k,
        \, 1 \leqslant i \leqslant n }}
        \prod_{i=1}^{n} \per(A[H_i | K_i]) = \\
        \sum_{\substack{(K_1^{(1)},K_1^{(2)}, K_2^{(1)}, K_2^{(2)},
        \ldots,K_n^{(1)}, K_n^{(2)}), \, K_i^{(1)} \subseteq (i-1)k +
        N_k, \, 1 \leqslant i \leqslant n,  \\ K_i^{(2)} \subseteq ik +
        N_k, \, 1 \leqslant i \leqslant n-1, \, K_n^{(2)}\subseteq N_k,
        \\ (K_i^{(1)} \cap K_i^{(2)}) \cap (K_j^{(1)} \cup
        K_j^{(2)}) =  \varnothing, \, 1 \leqslant i <j \leqslant n,
        |K_i^{(1)}| + |K_i^{(2)} | = k, \, 1 \leqslant i \leqslant n }}
        \prod_{i=1}^{n} \per(A[H_i | K_i^{(1)} \cup K_i^{(2)}])
    \end{multline}

    We show that if $K_i^{(1)} \subseteq (i-1) k + N_k$, \, $1 \leqslant i
    \leqslant n$, \, $K_i^{(2)} \subseteq ik + N_k$, $1 \leqslant i
    \leqslant n-1,$ \, $K_n^{(2)} \subseteq N_k$, \, $(K_i^{(1)}\cup
    K_i^{(2)}) \cap (K_j^{(1)}\cup K_j^{(2)}) = \varnothing$, $1 \leqslant i
    <j \leqslant n$, $|K_i^{(1)}| + |K_i^{(2)}| = k,$\ $1\leqslant i
    \leqslant n$, then
    $|K_1^{(1)}|=|K_2^{(1)}|=\ldots=|K_n^{(1)}|$. Let $1\leqslant
    i\leqslant n-1$. Then $|K_i^{(1)}|+|K_i^{(2)}|=k$,
    $K_{i+1}^{(1)} \subseteq ik + N_k,$ $K_i^{(2)} \subseteq ik + N_k$, \,
    $K_{i+1}^{(1)} \cap K_i^{(2)} = \varnothing $ and therefore
    $K_{i+1}^{(1)} \subseteq (ik + N_k) \setminus K_i^{(2)}$. It
    follows that
    $|K_{i+1}^{(1)}| \leqslant k - |K_i^{(2)}| = |K_i^{(1)}|$, i.~d.
    $|K_i^{(1)}| \geqslant |K_{i+1}^{(1)}|$ for all $i$, \, $1 \leqslant i
    \leqslant n-1$. So $|K_1^{(1)}| \geqslant |K_2^{(1)}| \geqslant
    \ldots \geqslant |K_n^{(1)}|$. Further, $|K_n^{(1)}| + |K_n^{(2)}| = k,$
    \, $K_1^{(1)} \subseteq N_k$, $K_n^{(2)} \subseteq N_k$, \,
    $K_{n}^{(1)} \cap K_n^{(2)} = \varnothing$, and therefore
    $K_n^{(2)}\subseteq N_k\setminus K_1^{(1)}$. Therefore $|K_n^{(2)}|
    \leqslant k-|K_1^{(1)}|$ and it follows that,
    $k -
    |K_n^{(1)}| = |K_n^{(2)}| \leqslant k - |K_1^{(1)}|  $, i.~d. $k -
    |K_n^{(1)}| \leqslant k - |K_1^{(1)}|$. Therefore, $|K_n^{(1)}|
    \geqslant |K_1^{(1)}|$. Hence and from inequalities $|K_1^{(1)}| \geqslant
    |K_2^{(1)}| \geqslant \ldots \geqslant |K_n^{(1)}|$ it follows
    that
    $|K_1^{1}|= |K_2^{(1)}| = \ldots = |K_n^{(1)}|$. We show that
    sets
    $K_1^{(2)}, K_2^{(2)}, \ldots, K_n^{(2)} $ be defined uniquely by choice
     of sets $K_1^{(1)}, K_2^{(1)}, \ldots, K_n^{(1)}$.
    Really, from inclusions $K_{i+1}^{(1)} \subseteq ik+ N_k$,
    $K_i^{(2)} \subseteq ik + N_k$ and equalities $|K_i^{(1)}| +
    |K_i^{(2)}| = k$, $|K_{i+1}^{(1)}| = |K_i^{(1)}|$ it follows
    that
    $K_i^{(2)} = (ik + N_k) \setminus K_{i+1}^{(1)}$ for all $i$,
    $1 \leqslant i \leqslant n-1$. From inclusions  $K_{1}^{(1)} \subseteq N_k$,
    $K_n^{(2)} \subseteq N_k$ and equalities $|K_n^{(1)}| +
    |K_n^{(2)}| = k$, $|K_1^{(1)}| = |K_n^{(1)}|$ it follows that
    $K_n^{(2)} =  N_k \setminus K_1^{(1)}$. From multilinear of the
    permanent of square matrix as function of columns it follows
    that
    \begin{equation}
        \per(A[H_i | K_i^{(1)} \cup K_i^{(2)}]) = a_0^{|K_i^{(1)}|}
        a_1^{K_i^{(2)}} \per (J_k) = k! a_0^{|K_i^{(1)}|} a_1^{k -
        |K_i^{(1)}|}
    \end{equation}
    Therefore
        \begin{multline}
        \per(A) = \sum_{\substack{(K_1^{(1)},K_2^{(1)}, \ldots,K_n^{(1)}), \,
        K_i^{(1)}
        \subseteq (i-1)k + N_{k}, \, 1 \leqslant i \leqslant n, \\
        |K_1^{(1)}| = |K_2^{(1)}| = \ldots =|K_n^{(1)}| }}
        \prod_{i=1}^{n} k! a_0^{|K_i^{(1)}|} a_1^{k- |K_i^{(1)}|} = \\
        =\sum_{l=0}^{k} \sum_{\substack{(K_1^{(1)}, K_2^{(1)},
        \ldots,K_n^{(1)}), \, K_i^{(1)} \subseteq (i-1)k +
        N_k, \, 1 \leqslant i \leqslant n,  \\
        |K_1^{(1)}| = |K_2^{(1)} | = \ldots = |K_n^{(1)}|  }}
        (k! a_0^l a_1^{k-l})^n  = \\
    =\sum_{l=0}^{k} \Card \{(K_1^{(1)}, K_2^{(1)},
        \ldots,K_n^{(1)} \mid \, K_i^{(1)} \subseteq (i-1)k +
        N_k, \, 1 \leqslant i \leqslant n,  \\
        |K_1^{(1)}| = |K_2^{(1)} | = \ldots = |K_n^{(1)}| = l \} \cdot
        (k! a_0^l a_1^{k-l})^n  = \\
        =\sum_{l=0}^k \binom{k}{l}^n (k! a_0^l a_1^{k-l})^n =
        \sum_{l=0}^k \binom{k}{l}^n (k! a_0^{k-l} a_1^l)^n.
    \end{multline}

    \begin{theorem}
        Let $n \geqslant 1$ and $t \geqslant 1$, let $a_0, a_t$~---
        be elements in a commutative ring with unity and let $d =
        \GCD(n,t)$. Then
        \begin{equation}
            \per  ( (a_0 I_n + a_t P_n^t) \otimes J_k) =
            [(k!)^{\frac{n}{d}} \sum_{l=0}^k
            \binom{k}{l}^{\frac{n}{d}}(a_0^{k-l}
            a_t^l)^{\frac{n}{d}} ]^d
        \end{equation}
    \end{theorem}
    \begin{proof}
        From [1, lemma~3, p.~12] it follows that the matrix $a_0 I_n + a_t
        P_n^t$ is permutation equivalent to direct sum of $d$
        matrices
        $a_0 I_{n/d} + a_t P_{n/d}$. Therefore
        \begin{multline*}
        \per( (a_0 I_n + a_t
        P_n^t) \otimes J_k) = [\per ( (a_0 I_{n/d} + a_t P_{n/d})
        \otimes J_k)]^d = [(k!)^{n/d} \sum_{l=0}^k
        \binom{k}{l}^{n/d} ( a_0^{k-l}a_t^l)^{n/d}]^d.
        \end{multline*}
    \end{proof}
    \begin{theorem}
\label{33} Let $a_0,a_1$~be elements in a commutative ring with
unity. Then for all $n \geqslant 1$ and $k \geqslant 1$
\begin{multline}
R(x;(a_0 I_n + a_1 P_n)\otimes J_k)= \\ =
\sum_{\substack{(l_1^{(1)},l_2^{(1)},\ldots,l_n^{(1)}), \,
(l_1^{(2)},l_2^{(2)},\ldots,l_n^{(2)})\in
\{0,1,\ldots,k\}^n \\
l_i^{(2)} \leqslant \min (k-l_i^{(1)}, k- l_{(i+1)\mod n}^{(1)}),\,
1 \leqslant i \leqslant n}} \biggl( \prod_{i=1}^n
\binom{k}{l_i^{(1)}} \binom{k - l_{(i+1)\mod n}^{(1)}}{l_i^{(2)}}
\binom{k}{l_i^{(1)}+ l_i^{(2)}} \times \\ \times \bigl( \bigl(
l_i^{(1)}+l_i^{(2)} \bigr)! \bigr) a_0^{\sum\limits_{i=1}^n
l_i^{(1)}} a_1^{\sum\limits_{i=1}^n l_i^{(2)}}
x^{\sum\limits_{i=1}^n (l_i^{(1)}+ l_i^{(2)})} \label{44}
\end{multline}
\end{theorem}
\begin{proof}
Let $A = (a_0 I_n + a_1 P_n) \otimes J_k$, $H_i = (i-1)k + N_k$, $1
\leqslant i \leqslant n$. Then from the lemma~\ref{lem:5} it follows
that
\begin{multline}
R(x;A)= \sum_{\substack{(K_1,K_2,\ldots,K_s), \, K_i \subset
(i-1)k+N_{2k},\,
 1 \leqslant i \leqslant n-1 \\ K_n \subset ((n-1)k + N_k)\cup N_k, \, K_i \cap K_j =
\varnothing, \, 1 \leqslant i<j \leqslant n, \, |K_i| \leqslant k,
\, 1 \leqslant i \leqslant n }} \biggl( \prod_{i=1}^s \per(A[H_i
\mid K_i]) \biggr)x^{\sum\limits_{i=1}^s|K_i|} = \\ =
\sum_{\substack{ (K_1^{(1)}, K_1^{(2)}, K_2^{(1)}, K_2^{(2)},\ldots,K_n^{(1)}, K_n^{(2)}), \\
K_i^{(1)} \subseteq (i-1)k + N_k, \, 1 \leqslant i \leqslant n, \\
K_i^{(2)} \subseteq ik + N_k, \, 1 \leqslant i \leqslant n-1, \, K_n^{(2)} \subseteq N_k \\
(K_i^{(1)} \cup K_i^{(2)})\cap (K_j^{(1)}\cup K_j^{(2)}) =
\varnothing, 1 \leqslant i < j \leqslant n,\\
|K_i^{(1)}|+|K_i^{(2)}|\leqslant k, \, 1 \leqslant i \leqslant n.}}
\biggl( \prod_{i=1}^n \per(A[H_i \mid K_i^{(1)} \cup
K_i^{(2)}])\biggr) x^{\sum\limits_{i=1}^n (|K_i^{(1)}|+
|K_i^{(2)}|)} \label{45}
\end{multline}
Let $l_i^{(1)} = |K_i^{(1)}|$, $l_i^{(2)} = |K_i^{(2)}|$, $1
\leqslant i \leqslant n$. Then from the inclusions $K_i^{(2)}
\subseteq ik+ N_k$, $K_{i+1}^{(1)} \subseteq ik + \\ +N_k$, $1
\leqslant i \leqslant n-1$, $K_1^{(1)} \subseteq N_k$, $K_n^{(2)}
\subseteq N_k$ it follows that $K_i^{(2)} \subseteq (ik+N_k)
\setminus K_{i+1}^{(1)}$, $1 \leqslant i \le n-1$,
$K_n^{(2)}\subseteq N_k \setminus K_1^{(1)}$ and therefore
$l_i^{(2)} \leqslant k - l_{(i+1) \mod n}^{(1)}$. Since the choice
of the sets $K_i^{(1)}$, $1 \leqslant i \leqslant n$, realized
independent from each other, then after selected sets
$K_1^{(1)},\ldots, K_n^{(1)}$, the sets $K_1^{(2)},\ldots,
K_n^{(2)}$ also choose independent from each other with regard
inclusions $K_i^{(2)}\subseteq (ik+ N_k) \setminus K_{i+1}^{(1)}$,
$1 \leqslant i \leqslant n-1$, $K_n^{(2)} \subseteq N_k \setminus
K_1^{(1)}$. From multilinear of the permanent of an $k \times l$,
rectangular matrix if $k \geqslant l$, as function of columns it
follows that
$$
\per (A[H_i \mid K_i^{(1)} \cup K_i^{(2)}]) = \binom{k}{|K_i^{(1)}|+
|K_i^{(2)}|} ((|K_i^{(1)}|+ |K_i^{(2)}|)!)\cdot a_0^{|K_i^{(1)}|}
a_1^{(K_i^{(2)})}.
$$
Therefore
\begin{multline}
R(x;A)= \sum_{\substack{(l_1^{(1)},l_2^{(1)},\ldots,l_n^{(1)}), \,
(l_1^{(2)},l_2^{(2)},\ldots,l_n^{(2)})\in
\{0,1,\ldots,k\}^n \\
l_i^{(2)} \leqslant \min (k-l_i^{(1)}, k- l_{(i+1)\mod n}^{(1)}),\,
1 \leqslant i \leqslant n}} \biggl( \prod_{i=1}^n
\binom{k}{l_i^{(1)}} \cdot \biggl( \prod_{i=1}^n \binom{k -
l_{(i+1)\mod n}^{(1)}}{l_i^{(2)}} \biggr) \times \\ \shoveright{
\times \biggl( \prod_{i=1}^n \binom{k}{l_i^{(1)}+ l_i^{(2)}} \bigl(
\bigl( l_i^{(1)}+l_i^{(2)} \bigr)! \bigr) a_0^{ l_i^{(1)}}
a_1^{l_i^{(2)}} \biggr) x^{\sum\limits_{i=1}^n (l_i^{(1)}+
l_i^{(2)})} =} \\ \shoveleft{=
\sum_{\substack{(l_1^{(1)},l_2^{(1)},\ldots,l_n^{(1)}), \,
(l_1^{(2)},l_2^{(2)},\ldots,l_n^{(2)})\in
\{0,1,\ldots,k\}^n \\
l_i^{(2)} \leqslant \min (k-l_i^{(1)}, k- l_{(i+1)\mod n}^{(1)}),\,
1 \leqslant i \leqslant n}} \biggl( \prod_{i=1}^n
\binom{k}{l_i^{(1)}} \binom{k - l_{(i+1)\mod n}^{(1)}}{l_i^{(2)}}
\binom{k}{l_i^{(1)}+ l_i^{(2)}} \times} \\ \shoveright{ \times
\bigl( \bigl( l_i^{(1)}+l_i^{(2)} \bigr)! \bigr)
a_0^{\sum\limits_{i=1}^n l_i^{(1)}} a_1^{\sum\limits_{i=1}^n
l_i^{(2)}} x^{\sum\limits_{i=1}^n (l_i^{(1)}+ l_i^{(2)})} =} \\
\shoveleft{= \sum_{\substack{(l_1^{(1)},l_2^{(1)},\ldots,l_n^{(1)}),
\, (l_1^{(2)},l_2^{(2)},\ldots,l_n^{(2)})\in \{0,1,\ldots,k\}^n }}
\biggl( \prod_{i=1}^n \binom{k}{l_i^{(1)}} \binom{k - l_{(i+1)\mod
n}^{(1)}}{l_i^{(2)}} \binom{k}{l_i^{(1)}+ l_i^{(2)}} \times} \\
\times \bigl( \bigl( l_i^{(1)}+l_i^{(2)} \bigr)! \bigr)
a_0^{\sum\limits_{i=1}^n l_i^{(1)}} a_1^{\sum\limits_{i=1}^n
l_i^{(2)}} x^{\sum\limits_{i=1}^n (l_i^{(1)}+ l_i^{(2)})},
\label{46}
\end{multline}
Since $l_i^{(2)} > \min (k-l_i^{(1)}, k - l_{(i+1)\mod n}^{(1)})$,
then
$$
\binom{k-l_{(i+1) \mod n}}{l_i^{(2)}} \binom{k}{l_i^{(1)}+
l_i^{(2)}}=0
$$
\end{proof}
\begin{lemma}
Let~$A$~be an~$m \times n$ rectangular matrix over a commutative
ring with unity and $m>n$. Let $(H_1,H_2,\ldots,H_s)$~be a fixed
ordered partition of the set~$N_m$. Then
\begin{gather}
\per(A) = \sum_{\substack{(K_1,K_2,\ldots,K_s), \, K_i \subseteq N_n,\\
|K_i| \leqslant |H_i|, \, 1 \leqslant i \leqslant s \\
|K_1|+|K_2|+\ldots+|K_s|=n, \, K_i \cap K_j = \varnothing, \, 1
\leqslant i < j \leqslant s}}
\prod_{i=1}^s \per(A[H_i \mid K_i]), \label{47}\\
R(x;A)= \sum_{\substack{(K_1,K_2,\ldots,K_s), \, K_i \subseteq N_n,\\
|K_i| \leqslant |H_i|, \, 1 \leqslant i \leqslant s \\ K_i \cap K_j
= \varnothing, \, 1 \leqslant i < j \leqslant s}} \biggl(
\prod_{i=1}^s \per(A[H_i \mid K_i])
\biggr)x^{\sum\limits_{i=1}^s|K_i|} \label{48}
\end{gather}
\label{lem:6}
\end{lemma}
\begin{lemma}
\label{lem:7} Let~$A$~be an~$m \times n$ rectangular matrix over a
commutative ring with unity and let~$H$~be a subset of the
set~$N_m$. Then
\begin{equation}
R(x;A) = \sum_{k \subseteq N_n, \, |K| \leqslant |H|} \per(A[H \mid
K]) x^{|K|} R(x ; A[N_m \setminus H \mid N_n \setminus K])
\label{49}
\end{equation}
\end{lemma}
\begin{proof}
From the lemmas~\ref{lem:5} and~\ref{lem:6} it follows that
\begin{multline}
R(x;A) = \sum_{k \subseteq N_n, \, |K| \leqslant |H|} \per(A[H \mid
K]) \sum_{L \subseteq N_n \setminus K, \, |L| \leqslant |N_m
\setminus H|} \per(A[N_m \setminus H \mid L])x^{|K|+|L|} = \\ =
\sum_{k \subseteq N_n, \, |K| \leqslant |H|} \per \bigl(A[H \mid
K]\bigr)  x^{|K|} \biggl( \sum_{L \subseteq N_n \setminus K, \, |L|
\leqslant |N_m \setminus H|} \per(A[N_m \setminus H \mid L])x^{|L|}
\biggr)= \\ = \sum_{k \subseteq N_n, \, |K| \leqslant |H|} \per (A[H
\mid K]) x^{|K|} R(x; A[N_m \setminus H \mid N_n \setminus K]).
\end{multline}
\end{proof}

Let $0 \leqslant r \leqslant kt$ and let $k \geqslant 1$, $t
\geqslant 0$. On the set $G_t^{[k]}$ we define a square matrix
$$
A_{r,t}^{[k]}(a_0,a_1,\ldots,a_t) =
\bigl(a_{\overline{\alpha},\overline{\beta}})_{\overline{\alpha},\overline{\beta}
\in G_t^{[k]}})
$$
as follows: $A_{0,0}^{[k]}(a_0) = k! a_0^k$. if $t \geqslant 1$ and
$\overline{\alpha} = (1^{\langle l_1 \rangle},2^{\langle l_2
\rangle}, \ldots, t^{\langle l_t \rangle}) \in G_{s,t}^{[k]}$. If $s
< r$, then
\begin{equation}
a_{\overline{\alpha},\overline{\beta}} =
\begin{cases}
(l_t + v)! \binom{k}{l_t+v} \binom{l_1}{p_1} \binom{l_2}{p_2}\ldots
\binom{l_{t-1}}{p_{t-1}} \binom{k}{v - \sum\limits_{i=1}^{t-1} p_i}
a_0^{v - \sum\limits_{i=1}^{t-1}p_i}
a_1^{p_1} a_2^{p_2} a_{t-1}^{p_{t-1}} a_t^{l_t}, \\
\text{if } \overline{\beta} = \biggl( 1^{\langle k- v +
\sum\limits_{i=1}^{t-1} p_i \rangle}, 2^{\langle l_1 - p_1 \rangle},
3^{\langle l_2 - p_2 \rangle}, \ldots, t^{\langle l_{t-1} - p_{t-1}
\rangle} \biggr), \,
0 \leqslant p_i \leqslant l_i \\
1 \leqslant i \leqslant t-1, \; \max\biggl(\sum\limits_{i=1}^{t-1}p_i, s-r+k- l_t\biggr) \leqslant v \leqslant k-l_t, \\
0, \quad \text{ in all other cases.}
\end{cases}
\label{50}
\end{equation}
If $s \geqslant r$, then
\begin{equation}
a_{\overline{\alpha},\overline{\beta}} =
\begin{cases}
k! \binom{l_1}{p_1} \binom{l_2}{p_2}\ldots \binom{l_{t-1}}{p_{t-1}}
\binom{k}{v - \sum\limits_{i=1}^{t-1} p_i} a_0^{v -
\sum\limits_{i=1}^{t-1}p_i}
a_1^{p_1} a_2^{p_2} a_{t-1}^{p_{t-1}} a_t^{k-v}, \\
\text{if } \overline{\beta} = \biggl( 1^{\langle k- v +
\sum\limits_{i=1}^{t-1} p_i \rangle},
2^{\langle l_1 - p_1 \rangle}, 3^{\langle l_2 - p_2 \rangle}, \ldots, t^{\langle l_{t-1} - p_{t-1} \rangle} \biggr), \\
1 \leqslant i \leqslant t-1, \; \max\biggl(\sum\limits_{i=1}^{t-1}p_i, k- l_t\biggr) \leqslant v \leqslant \min(k, s-r+k-l_t), \\
0, \quad \text{ in all other cases.}
\end{cases}
\label{51}
\end{equation}
\begin{theorem}
Let $t \geqslant 0$ and let $a_0,a_1,\ldots,a_t$~be an elements in a
commutative ring with unity, let $n \geqslant 1$, $k \geqslant 1$,
let $0 \leqslant r \leqslant k \min (n,t)$, let $\overline{\beta}
\in G_{r, \min(n,t)}^{[k]}$, $m = \min(n,t)$. Then for all
$\overline{\alpha} \in G_t^{[k]}$
\begin{multline}
\per \biggl(\biggl(\biggl( \sum_{i=0}^t a_i T_{n+t}^{(i)} \biggr)
\otimes J_k \biggr) \bigl[ N_{nk} \mid k \bigl(1^{\langle k -
m_{\overline{\beta}(1)} \rangle},
2^{\langle k - m_{\overline{\beta}(2)} \rangle}, \ldots, m^{\langle k - m_{\overline{\beta}(m)} \rangle}, \\
\shoveright{ mk+1, mk+2, \ldots, nk, nk + k \overline{\alpha}
\bigr)\bigr] \biggr) =} \\ = \biggl( \prod_{i=1}^{\min(n,t)}
\binom{k}{m_{\overline{\beta}}(i)}^{-1} \biggr)
\sum_{\substack{\overline{\gamma} \in G_t^{[k]} \\ \{
\overline{\gamma} \} \supseteq \{ \overline{\beta}\} }}
\prod_{i=1}^{\min(n,t)}
\binom{m_{\overline{\gamma}}(i)}{m_{\overline{\beta}}(i)} \cdot
\bigl(A_{r,t}^{[k]}(a_0,a_1,\ldots,a_t) \bigr)^n
[L(\overline{\alpha})\mid L (\overline{\gamma})] \label{52}
\end{multline}
\end{theorem}
From lemmas~\ref{lem:5} and~\ref{lem:6} it follows that
\begin{lemma}
\label{lem:8} Let~$A$~be an~$m \times n$ rectangular matrix over a
commutative ring with unity. Let $(H_1,H_2,\ldots,H_s)$~be a fixed
ordered partition of the set~$N_m$. Then
\begin{gather}
\per(A) = \sum_{\substack{(K_1,K_2,\ldots,K_s), \, K_i \subseteq N_n,\\
|K_i| \leqslant |H_i|, \, 1 \leqslant i \leqslant s \\ K_i \cap K_j
= \varnothing, \, 1 \leqslant i < j \leqslant s, \,
\sum\limits_{i=1}^s |K_i| = \min(m,n)}}
\prod_{i=1}^s \per(A[H_i \mid K_i]), \label{53}\\
R(x;A)= \sum_{\substack{(K_1,K_2,\ldots,K_s), \, K_i \subseteq N_n,\\
|K_i| \leqslant |H_i|, \, 1 \leqslant i \leqslant s \\ K_i \cap K_j
= \varnothing, \, 1 \leqslant i < j \leqslant s}} \biggl(
\prod_{i=1}^s \per(A[H_i \mid K_i])
\biggr)x^{\sum\limits_{i=1}^s|K_i|} \label{54}
\end{gather}
\end{lemma}

\end{document}